\documentclass[10pt]{article}
\usepackage{amscd,enumerate}
\usepackage{amsmath,amsfonts}
\usepackage{color}
\usepackage{graphicx}
\usepackage{times}
\usepackage{amsthm}
\usepackage{amssymb}
\usepackage{xspace}
\usepackage{xcolor}
\usepackage{comment}
\usepackage[colorlinks=true]{hyperref}
\usepackage{pgfplots}

\newtheorem{theorem}{Theorem}
\newtheorem{proposition}[theorem]{Proposition}
\newtheorem{lemma}[theorem]{Lemma}
\newtheorem{corollary}[theorem]{Corollary}
\theoremstyle{definition}
\newtheorem{assumption}[theorem]{Assumption}
\newtheorem{remark}[theorem]{Remark}
\newtheorem{example}[theorem]{Example}
\newtheorem{algorithm}[theorem]{Algorithm}

\title{New results for variational regularization with oversmoothing penalty term in Banach spaces}

\author{
Bernd Hofmann\footnotemark[1] \and
Chantal Klinkhammer\footnotemark[2]
\and Robert Plato\footnotemark[2]}
\newcommand{\para}{\alpha}

\newcommand{\parb}{\beta}
\newcommand{\pardel}{\para_*}
\newcommand{\pardelb}{\gamma_\delta}
\newcommand{\ix}{\mathcal{X}}
\newcommand{\iz}{\mathcal{Z}}
\newcommand{\yps}{\mathcal{Y}}
\newcommand{\pp}{p}
\newcommand{\qq}{q}

\newcommand{\fdelta}{f^\delta}
\newcommand{\reza}{ \mathbb{R} }

\newcommand{\Landau}{\mathcal{O}}
\newcommand{\refeq}[1]{(\ref{eq:#1})}

\newcommand{\G}{G}

\newcommand{\norm}[1]{\Vert \hspace{0.4mm} #1 \hspace{0.4mm} \Vert}

\newenvironment{myenumerate}{%
\begin{list}{(\alph{enumcount})}
{\setcounter{enumcount}{1}\usecounter{enumcount}
\setlength{\topsep}{1mm}
\setlength{\itemsep}{0mm}
\setlength{\labelwidth}{0mm}
\setlength{\labelsep}{1mm}
\setlength{\itemindent}{1mm}
\setlength{\leftmargin}{0mm}
}}{\end{list}}

\newcommand{\DF}{\Dset(\F)}

\newcommand{\Dset}{\mathcal{D}}

\newcommand{\myu}{u}

\newcommand{\ust}{u^\dagger}

\newcommand{\fst}{f^\dagger}
\newcommand{\fdel}{f^\delta}

\newcommand{\upardeldel}{u_{\pardel}^\delta}
\newcommand{\upardeldelb}{u_{\pardelb}^\delta}
\newcommand{\upardelst}{u_{\pardel}^\delta}

\newcommand{\tikreg}{Tikhonov regularization\xspace}

\newcommand{\tinybullet}{{\tiny \raisebox{0.6mm}{$ \bullet $}}}
\newenvironment{mylist}{%
\begin{list}{\tinybullet}
{\setlength{\topsep}{0.2cm}
\setlength{\itemsep}{0mm}
\setlength{\labelwidth}{0mm}
\setlength{\labelsep}{3mm}
\setlength{\itemindent}{3mm}
\setlength{\leftmargin}{0mm}
}}{\end{list}}

\newenvironment{mylist_indent}{%
\begin{list}{\tinybullet}
{\setlength{\topsep}{0.2cm}
\setlength{\itemsep}{0mm}
\setlength{\labelwidth}{2mm}
\setlength{\labelsep}{3mm}
\setlength{\itemindent}{0mm}
\setlength{\leftmargin}{5mm}
}}{\end{list}}

\newcommand{\kla}[1]{(#1)}
\newcommand{\mfrac}[2]{\dfrac{\mbox{\footnotesize \raisebox{-0.5mm}{$#1$}}}%
{\mbox{\footnotesize \raisebox{0.8mm}{$#2$}}}}
\newcommand{\proofend}{\qquad \endproof}

\newcommand{\Landauno}[1]{\Landau\kla{#1}}
\newcommand{\as}{\quad \text{as} \ \ }

\newcommand{\R}{\mathcal{R}}

\newcommand{\remarkend}{\quad \ensuremath{\vartriangle}}

\newcommand{\defeq}{:=}
\newcommand{\for}{\quad \text{for} \ \ }

\newcommand{\ints}[4]{\mathop{\raisebox{-0.1mm}{\mbox{\Large$ \textstyle \int $}}}\nolimits_{\hspace{-1mm}#1}^{#2} \hspace{-0.2mm} #3 \, #4}

\newcommand{\rhs}{right-hand side\xspace}

\newenvironment{myenumerate_indent}{%
\begin{list}{(\alph{enumcount})}
{\setcounter{enumcount}{1}\usecounter{enumcount}
\setlength{\topsep}{1mm}
\setlength{\itemsep}{0mm}
\setlength{\labelwidth}{5mm}
\setlength{\labelsep}{2mm}
\setlength{\itemindent}{-0mm}
\setlength{\leftmargin}{7mm}
}}{\end{list}}

\newcommand{\myubar}{\overline{u}}

\newcommand{\inset}[1]{\{ \, #1 \, \}}

\newcommand{\ualpaux}[1][\parb]{\widehat{u}_{#1}}
\newcommand{\ualpauxst}[1][\betast]{\widehat{u}_{#1}}

\newcommand{\fab}{r(a+1)}
\newcommand{\norma}[1]{\norm{#1}_{-a}}

\newcommand{\normone}[1]{\norm{#1}_{1}}
\newcommand{\normmone}[1]{\norm{#1}_{-1}}

\newcommand{\normyps}[1]{\norm{#1}}
\newcommand{\normypsqua}[1]{\norm{#1}^\myr}
\newcommand{\normix}[1]{\norm{#1}}
\newcommand{\landau}{o}

\newcommand{\landauno}[1]{\landau\kla{#1}}
\newcommand{\lfrac}[2]{#1/#2}

\newcommand{\ca}{c_a}
\newcommand{\cb}{C_a}
\newcommand{\upardel}[1][\para]{{u}_{#1}^\delta}

\newcommand{\updd}{\upardel[\pardel]}
\newcommand{\upddb}{\upardel[\pardelb]}
\newcommand{\udel}{\updd}

\newcommand{\tikfun}{T}
\newcommand{\pad}[2][\para]{\tikfun_{#1}^\delta(#2)}

\newcommand{\Jdel}[1]{\norm{\myF{#1} - \fdel}}

\newcommand{\myomega}[1]{\normonequa{#1}}

\newcommand{\F}[1]{\Fpur#1}
\newcommand{\Fpur}{F}
\newcommand{\myb}{b}
\newcommand{\myc}{c}

\newcommand{\parst}{\para_*}

\newcommand{\ubar}{\overline{u}}

\newcommand{\mainassump}{Let Assumption \ref{th:main_assump} be satisfied.\xspace}
\newcommand{\mainassumpb}{Let Assumption \ref{th:main_assump} be satisfied\xspace}

\newcommand{\nondecreasing}{non-decreasing\xspace}
\newcommand{\monotonically}{}

\newcommand{\normonequa}[1]{\norm{#1}_{1}^\myr}
\newcommand{\myr}{r}
\newcommand{\myrr}{\tau}

\newcommand{\Jdelqua}[1]{\norm{F(#1) - \fdel}^\myr}
\newcommand{\cp}{c_p}
\newcommand{\fracpower}{fractional power\xspace}
\newcommand{\fracpowers}{\fracpower{}s\xspace}
\newcommand{\domain}{\mathcal{D}}
\newcommand{\range}{\mathcal{R}}
\newcommand{\mykap}{\kappa}
\newcommand{\betast}{\beta_\delta}
\newcommand{\RG}{\overline{\R(\G)}}
\newcommand{\linf}{L^\infty(0,1)}
\newcommand{\lone}{L^1(0,1)}
\newcommand{\myF}[1]{F(#1)}
\newcommand{\normmax}[1]{\norm{#1}_{\infty}}
\newcommand{\er}{e_r}

\newcommand{\ixtwo}{\overline{\R(G)}}
\newcommand{\postype}{non-negative type\xspace}
\newcommand{\Postype}{Non-negative type\xspace}
\newcommand{\loginv}[2][1]{\kla{\log#2}^{-#1}}

\newcommand{\mya}{a}

\newcommand{\myfa}{\psi}
\newcommand{\mychi}[2]{\chi_{#1,#2}}
\newcommand{\mychiinv}[2]{\chi^{-1}_{#1,#2}}

\newcommand{\mylog}{\log}
\newcommand{\mylogb}[1]{\log \tfrac{1}{#1}}
\newcommand{\myphi}{\varphi}
\newcommand{\myphib}[1]{\kla{\log \tfrac{1}{#1}}^{-1}}
\newcommand{\myphibrez}[1]{\log \tfrac{1}{#1}}

\begin{document}
%%\date{9 March 2023}
%%\date
\date{}
\maketitle

\pgfplotsset{compat = newest}
\pgfplotstableread{2023_03_06_Plot_data_UMIN.dat}{\mytable}
\renewcommand{\thefootnote}{\fnsymbol{footnote}}
\footnotetext[1]{Faculty of Mathematics, Chemnitz University of Technology, 09107 Chemnitz, Germany.}
\footnotetext[2]{Department of Mathematics, University of Siegen,
Walter-Flex-Str.~3, 57068 Siegen, Germany}

\newcounter{enumcount}
\newcounter{enumcountroman}
\renewcommand{\theenumcount}{(\alph{enumcount})}
\bibliographystyle{plain}
\begin{abstract}
In this article on variational regularization for ill-posed nonlinear problems, we are once again discussing the consequences of an oversmoothing penalty term. This means in our model that
the searched-for solution of the considered nonlinear operator equation does not belong to the domain of definition of the penalty functional. In the past years, such variational regularization has been investigated comprehensively in Hilbert scales, but rarely in a Banach space setting. Our
present results try to establish a theoretical justification of oversmoothing regularization in  Banach scales. This new study includes convergence rates results for a priori and a posteriori choices of the regularization parameter, both for H\"older-type smoothness and low order-type smoothness.
An illustrative example is intended to indicate the specificity of occurring non-reflexive Banach spaces.
\end{abstract}

\bigskip

{\parindent0em {\bf Keywords:}
Nonlinear ill-posed problem, variational regularization, oversmoothing penalty, convergence rates results,
a priori parameter choice strategy, discrepancy principle, logarithmic source conditions}

\smallskip

{\parindent0em {\bf AMS subject classifications:} 47J06, 65J20, 47A52}

\section{Introduction}
\label{intro}
The goal of this paper is the theoretical justification of variational regularization with oversmoothing penalties for nonlinear ill-posed problems in Banach scales. Precisely, we consider operator equations of the form
\begin{equation} \label{eq:opeq}
\F(u) = \fst \,,
\end{equation}
where $ \F: \ix \supset \DF \to \yps $ is a nonlinear operator
between infinite-dimensional Banach spaces $ \ix $ and $ \yps $ with norms $\|\cdot\|$.
We suppose that the \rhs $ \fst \in \yps $ is approximately given as $ \fdelta \in \yps$ satisfying the deterministic noise model
\begin{equation} \label{eq:noise}
\normyps{ \fdelta-\fst } \le \delta,
\end{equation}
with the noise level $\delta \ge 0$.
Throughout the paper, it is assumed that the considered equation~\eqref{eq:opeq} has a solution $ \ust \in \DF $ and is, at least at $ \ust$, locally ill-posed in the sense of \cite[Def.~1.1]{HofSch98} and \cite[Def.~3]{HofPla18}.
Consequently, a variant of regularization is required for finding stable approximations to the solution  $ \ust \in \DF $ of equation \eqref{eq:opeq}, and we exploit in this context a variant of variational regularization with regularization parameter $\para>0$, where the regularized solutions $\upardel$ are minimizers of the extremal problem
\begin{equation} \label{eq:TR}
\pad{\myu} \defeq \Jdelqua{\myu} + \para \myomega{\myu-\ubar} \to \min\,, \quad \textup{subject to} \quad  \myu \in \DF,
\end{equation}
with some exponent $ r > 0 $ being fixed. Here, $ \normone{\cdot} $ is a norm on a densely defined subspace $\ix_1$ of $\ix$, which is stronger than the original norm $ \normix{\cdot} $ in $ \ix $.
Precisely, we define the stronger norm $\normone{\cdot}$ by
$ \normone{u} = \norm{G^{-1}u}, u \in \R(G) $, where the generator
$ G: \ix \to \ix $ with range $\R(G)$ is a bounded linear operator, which is one-to-one and has an unbounded inverse $ G^{-1} $. Further conditions on the operator $ G $ are given in Section~\ref{sec:preparations} below.
Moreover, the element $\ubar \in \ix_1 \cap\DF$, occurring in the penalty term of the Tikhonov functional $T_\alpha^\delta$, denotes an initial guess.
Note that we restrict our consideration in this study to identical exponents $r$ for the misfit term and the penalty functional in order to avoid unnecessary technical complications.

In the present work, we discuss the nonlinear Tikhonov-type regularization \eqref{eq:TR} with focus on an oversmoothing penalty term. This means in our model that we have $ \ust \not \in \ix_1 $, or in other words $\|\ust\|_1=+\infty,$ which is an expression of
`non-smoothness' of the solution $\ust$ with respect to the reference Banach space $\ix_1$. Variational regularization of the form \eqref{eq:TR} with $r=2$ and oversmoothing penalty for nonlinear ill-posed operator equations \eqref{eq:opeq} has been investigated comprehensively in the past four years in Hilbert scales, and we refer to \cite{Hofmann_Mathe18,HofPla20} as well as further to the papers \cite{GHH20,HofHof20,HHMP22,HofMat20,KliPla22}. For related results on linear problems, see, e.g.,
\cite{Natterer84}.
%% and more recently \cite{George[08],Nair_Pereverzev_Tautenhahn[05]}.
Our present study continues and extends, along the lines of \cite{HofPla20}, the investigations on nonlinear problems to Banach scales. This new study includes fundamental error estimates
%, cf.\refeq{fin} below, and
yielding convergence and convergence rates results for a priori and a posteriori choices  of the regularization parameter,
both for H\"older-type smoothness and low order smoothness.
The necessary tools for low order smoothness in the Banach space setting are provided.
In addition, a relaxed nonlinearity and smoothing condition on the operator $F$
is considered that turns out to be useful for maximum norms.

Banach space results for the discrepancy principle in a pure equation form have already been proven for the oversmoothing case in the recent paper \cite{CHY21}. In parallel, such results have been developed for oversmoothing subcases to variants of $\ell^1$-regularization and sparsity promoting wavelet regularization in \cite[Sec.~5]{Miller21} and \cite{MillerHohage22}.

The outline of the remainder is as follows: in Section~\ref{sec:preparations} we summarize prerequisites and assumptions for the main results in the sense of error estimates and convergence rates for the regularized solutions. On the one hand, error estimates for a priori choices of the regularization parameter are presented in Section~\ref{sec:apriori_parameter_choices}. On the other hand, Section~\ref{sec:discrepancy_principle} presents results and consequences for using a discrepancy principle. An illustrative example in Section~\ref{sec:example} is intended to indicate the specificity of occurring non-reflexive Banach spaces. A numerical case study will be presented in Section~\ref{sec:study} that illustrates the theoretical results.
Technical details, constructions and verifications for proving the main results of the paper are given in the concluding Section~\ref{sec:verifications}.

%

%%}
%
\section{Prerequisites and assumptions}
\label{sec:preparations}
In this section, we introduce a scale of Banach spaces generated by an operator of positive type. Moreover, we define the logarithm of a positive operator and formulate the basic assumptions for our study.
The concluding subsection is devoted to well-posedness and stability assertions for the variant of variational regularization under consideration in this paper.

\subsection{\Postype operators, fractional powers, and regularization operators}
\label{postype_operators}
Let $ \ix $ with norm $\norm{\cdot}$ be a Banach space and $\mathcal{L}(\ix)$ with norm $\norm{\cdot}_{\scriptscriptstyle \mathcal{L}(\ix)}$ the associated space of bounded linear operators mapping in $\ix$. Furthermore, let the injective operator $ G \in \mathcal{L}(\ix)$ with range $\range(G)$ and unbounded inverse $G^{-1}$ be of \postype, i.e.,
\begin{align}
& G + \parb I: \ix \to \ix \ \textup{one-to-one and onto},
\quad \norm{(G + \parb I)^{-1}}_{\scriptscriptstyle \mathcal{L}(\ix)} \le \frac{\kappa_*}{\parb}, \quad \parb > 0,
\label{eq:postype}
\end{align}
for some finite constant $ \kappa_* > 0 $.
Fractional powers of \postype operators may be defined as follows
\cite{Balakrishnan59,Balakrishnan60}:
\begin{myenumerate_indent}
\item
For $ 0 < \pp < 1 $, the fractional power $ G^\pp : \ix \to \ix $ is defined by
\begin{align}
G^\pp u \defeq \mfrac{\sin \pi \pp }{\pi}
\ints{0}{\infty}{s^{\pp-1} (G + sI)^{-1} G u }{ ds}
\for u \in \ix.
\label{eq:frac-power}
\end{align}
This defines a bounded linear operator on $ \ix $.
%%% which is one-to-one.

\item
For arbitrary values $ \pp \ge 1 $,
the bounded linear operator $  G^\pp : \ix \to \ix $ is defined by
$$
G^\pp  \defeq G^{\pp - \lfloor \pp \rfloor} G^{\lfloor \pp \rfloor}.
$$
We moreover use the notation $ G^0 = I $.
\end{myenumerate_indent}
In what follows, we shall need
the interpolation inequality for \fracpowers of operators, see, e.g.,
\cite{Komatsu66} or
\cite[Proposition 6.6.4]{Haase06}:
for each pair of real numbers \linebreak $ 0 < \pp < \qq $,
there exists some finite constant $ c =c(\pp,\qq) > 0 $ such that
\begin{align}
\norm{ \G^\pp u }
\le c \norm{\G^\qq u}^{\pp/\qq} \norm{ u }^{1-\pp/\qq}
\for u \in \ix.
\label{eq:interpol2}
\end{align}
For $ 0 < \pp < 1 = \qq $, the value of the constant can be chosen as $ c = 2(\kappa_* + 1)$, cf., e.g., \cite[Corollary 1.1.19]{Plato95}.
Under the stated assumptions on $G$, for each $\pp > 0 $, the fractional power
$ G^{\pp} $ is one-to-one, and we use the notation
$ G^{-\pp} = (G^\pp)^{-1} $.
We do not need that the operator $ G $ has dense range in $ \ix $.

The scale of normed spaces $\{\ix_\tau\}_{\tau \in \mathbb{R}}$,
generated by $\G$, is given by the formulas
\begin{align}
& \ix_\tau=\range(\G^\tau) \ \textup{ for }\, \tau > 0, \qquad
\ix _\tau= \ix \ \textup{ for }\, \tau \le 0, \nonumber \\
& \|u\|_\tau:=\|\G^{-\tau} u\| \ \textup{ for }\, \tau \in \mathbb{R}, \ u \in \ix_\tau.
\label{eq:taunorm}
\end{align}
For $ \tau < 0 $, topological completion of the spaces
$\ix_\tau = \ix$ with respect to the norm $ \|\cdot\|_\tau $ is not needed in our setting.
We note that $ (\G^p)_{p\ge 0} $ defines a $ C_0 $-semigroup on
$ \ixtwo $,
which in particular means that $ \G^p \myu \to \myu $ for $ p \downarrow 0 $ is valid for any
$ u \in \overline{\R(G)} $ (cf.~\cite[Proposition 3.1.15]{Haase06}). Finally, we note that
\begin{equation} \label{eq:chain1}
\R(G^{\tau_2}) \subset\R(G^{\tau_1}) \subset \overline{\R(\G)} \quad \mbox{for all} \;0 < \tau_1 < \tau_2 < \infty.
\end{equation}
\subsection{The logarithm $ \mathbf{\log G } $}
For the consideration of low order smoothness, we need to introduce the logarithm of
$ \G $. For selfadjoint operators in Hilbert spaces this can be done by spectral analysis, and we refer in this context for example to \cite{Hohage00}.
%%%,Mahale_Nair[07]}.
In Banach spaces, $ \log G $ may be defined as the infinitesimal generator of the $ C_0 $-semigroup $ (\G^p)_{p \ge 0} $ considered on $ \overline{\R(G)} $:
\begin{align*}
(\log \G) \myu & = \lim_{p \downarrow 0} \tfrac{1}{p}\kla{\G^p \myu-\myu},
\quad \myu  \in \domain(\log \G),
\end{align*}
where
\begin{align*}
\domain(\log\G) & =
\inset{ \myu \in \ix : \lim_{p \downarrow 0} \tfrac{1}{p}\kla{\G^p \myu- \myu}
\ \textup{exists} },
\end{align*}
cf.,~e.g.,~\cite{Nollau69} or \cite[Proposition 3.5.3]{Haase06}.
Low order smoothness of an element $ \myu \in \ix $ by definition then means
$ \myu \in \domain(\log \G) $. Note that
we obviously have $ \domain(\log \G) \subset \RG $. In addition, $ \R(\G^p) \subset \domain(\log \G)$ is valid for arbitrarily small $p>0$,
which follows from \cite[Satz 1]{Nollau69}. Summarizing the above notes, we have a chain of subsets of $\ix$ as
\begin{equation} \label{eq:chain2}
\R(\G^p) \subset \domain(\log \G)  \subset \overline{\R(\G)} \quad \mbox{for all} \;p>0.
\end{equation}
This means that also in the Banach space setting, any H\"older-type smoothness is stronger than low order smoothness.
\subsection{Main assumptions}
In the following assumption, we briefly summarize the structural properties of the space $\ix$, of the operator $F$ and of its domain $\DF$, in particular with respect to the solution $\ust$ of the operator equation. Moreover, we make one more assumption concerning $G$ in addition to
the requirements on the operator $G$ stated above.
%%%\eqref{eq:TR}.
%
\begin{assumption}
\label{th:main_assump}
\begin{myenumerate_indent}

\item
\label{it:spaces}
The infinite-dimensional Banach space $\ix$ has a separable pre-dual space $\iz$ with $\iz^*=\ix$ such that a weak$^*$-convergence denoted as $\rightharpoonup^*$ takes place in $\ix$.

\item The operator $ \F: \ix \supset \DF \to \yps $ is weak$^*$-to-weak sequentially continuous, i.e., for elements $u_n,u_0 \in \DF$, weak$^*$-convergence $u_n \rightharpoonup^*u_0$ in $\ix$ implies weak convergence $F(u_n) \rightharpoonup F(u_0)$ in $\yps$.

\item
The domain of definition $ \DF \subset \ix $ is a sequentially weak$^*$-closed subset of $\ix$.
\item
Let $\Dset \defeq \DF \cap \ix_1 \neq \varnothing $.

\item
Let the solution $ \ust \in \DF $ to equation \eqref{eq:opeq} with \rhs $\fst$ be
an interior point of the domain $ \DF $.

\item
Let the data $ \fdelta \in \yps$ satisfy the noise model \eqref{eq:noise}, and let the
initial guess $\ubar$ satisfy $ \ubar \in \ix_1 \cap \DF$.

\item
\label{it:normequiv_b}
Let $ a > 0 $, and let $ 0 < \ca \le \cb $ and $ c_0, c_1 > 0 $
be finite constants such that the following holds:
\begin{mylist}
\item For each $ u \in \Dset $ satisfying $ \norma{u-\ust} \le c_0 $, we have
\begin{align}
\label{eq:normequiv_a}
\normyps{\myF{\myu} - \fst} \le \cb\norma{\myu - \ust}.
\end{align}

\item
For each $ u \in \Dset $ satisfying $ \normyps{\myF{\myu} - \fst} \le c_1 $, we have
\begin{align}
\label{eq:normequiv_b}
\ca\norma{\myu - \ust} \le \normyps{\myF{\myu} - \fst}.
\end{align}
\end{mylist}

\item
\label{it:Gnew}
The operator $G: \ix \to \ix$ defined above possesses a pre-adjoint operator $\tilde G: \iz \to \iz$ such that $\tilde G^*=G$ holds true.

\end{myenumerate_indent}
\end{assumption}
%
\begin{comment}
\begin{remark} \label{th:rem-assump}
Items \ref{it:G-compact} and
Item \ref{it:G-preadjoint} are needed for the proof our version of well-posedness (existence and stability of regularized solutions in the norm of $ \ix $) in the sense of Theorem \ref{th:tifu-well-posed} below. Note that the preadjoint operator $\tilde G $ is necessarily also compact, see, e.g.,
\cite{Yosida[80]}.
\end{remark}
\end{comment}
%%
\begin{remark} \label{rem:rem1}
From the inequality \eqref{eq:normequiv_b} of item \ref{it:normequiv_b} in Assumption~\ref{th:main_assump}, we have for $\ust \in \ix_1$ that $\ust$ is the uniquely determined solution to equation \eqref{eq:opeq} in the set $\Dset$. For $\ust \notin \ix_1$, there is no solution at all to \eqref{eq:opeq} in $\Dset$. But in both cases, alternative solutions $u^* \notin \ix_1$ with $u^* \in \DF$ and $F(u^*)=\fst$ cannot be excluded in general. Note that the concept of penalty-minimizing solutions, which is usual in theory of Tikhonov regularization, does not make sense in the
case of oversmoothing penalties.
%However, there is an exception if $u^*$ is an interior point of $\DF$. Then a solution  $u^* \in \overline{\R(G)}$ to equation \eqref{eq:opeq} with right-hand side $\fst$ satisfies with $u=u^*$ the inequality \eqref{eq:normequiv_b}. This is a
%consequence of the continuity of $F$ from item \ref{it:conda} of Assumption~\ref{th:main_assump}.
\end{remark}

\subsection{Existence and stability of regularized solutions}

The following two propositions on existence and stability can be immediately taken from \cite[\S~4.1.1]{Schusterbuch12}, and we refer in this context also to \cite{HKPS07} and \cite[\S~3.2]{Scherzetal09}.

\begin{proposition}[well-posedness, cf.~Prop.~4.1 of \cite{Schusterbuch12}]
For all $\alpha>0$ and  $f^\delta \in \yps$, there exists a regularized solution $\upardel \in \mathcal{D}$, minimizing the Tikhonov functional $\pad{\myu}$ in \eqref{eq:TR} over all $u \in \DF$.
\end{proposition}

\begin{proposition}[stability, cf.~Prop.~4.2 of \cite{Schusterbuch12}]
For all $\alpha>0$, the minimizers $\upardel \in \mathcal{D}$ of the extremal problem \eqref{eq:TR} are stable with respect to the data $f^\delta$. More precisely, for a data sequence $\{f_n\}$ converging to $f^\delta$ with respect to the norm topology in $\yps$,
i.e.~$\lim \limits_{n \to \infty}\|f_n-f^\delta\|=0$, every associated sequence $\{u_n\}$ of minimizers to the extremal problem
$$\|F(u)-f_n\|^r + \para \myomega{\myu-\ubar} \to \min\,, \quad \textup{subject to} \quad  \myu \in \DF,  $$
has a subsequence $\{u_{n_k}\}$, which converges in the weak$^*$-topology of $\ix$, and the weak$^*$-limit $\tilde u$ of each such subsequence is a minimizer $\upardel$ of \eqref{eq:TR}.
\end{proposition}

In order to prove the applicability of both propositions to our situation, we have to state that the relevant items of Assumptions 3.11 and 3.22 in \cite{Schusterbuch12} can be met as a consequence of our Assumption~\ref{th:main_assump} by taking into account Remark 4.9 in \cite{Schusterbuch12}, where the transfer from the weak-situation to the weak$^*$-situation is explained. In particular, since $G$ is bounded, we have $\|u\|_1 \ge \tilde c \|u\|$ with some constant $\tilde c>0$ for all $u \in \ix_1$.
Then the penalty functional $\Omega: \ix \to [0,\infty]$ of the Tikhonov functional $T_\alpha^\delta$,  defined as
$$ \Omega(u):= \left\{\begin{array}{lcl} \|u-\ubar\|_1^r=\|G^{-1}(u-\ubar )\|^r & \mbox{for} & u \in \ix_1\\ +\infty & \mbox{for} & u \in \ix \setminus \ix_1\end{array} \right.\,,$$
possesses the required stabilizing property (cf.~item (c) of Assumption 3.22 in \cite{Schusterbuch12}) as a consequence of the sequential Banach--Alaoglu theorem, which implies that the sublevel sets $$\{u \in \ix: \Omega(u) \le c\} \subset \{u \in \ix: \|u-\ubar \|^r \le \frac{c}{\tilde c^r}\} $$ are weak$^*$ sequentially compact in $\ix$ for all $c \ge 0$.
Moreover, $\Omega$ is sequentially weak$^*$ lower semi-continuous (cf.~item (b) of Assumption 3.22 in \cite{Schusterbuch12}), because the existence of a pre-adjoint operator $\tilde G: \iz \to \iz$ to $G$ in the sense of item \ref{it:Gnew} of Assumption~\ref{th:main_assump} ensures that
$G: \ix \to \ix$ is weak$^*$-to-weak$^*$ sequentially continuous. This together with the weak$^*$ lower semi-continuity of the norm functional in $\ix$ yields the sequentially weak$^*$ lower semi-continuity of the penalty functional $\Omega$.
%
%
%
%%%\begin{remark} \label{rem:alter}
%%%The weak$^*$ lower semi-continuity of $\Omega$ is also valid if we replace the explicit requirement of item \ref{it:Gnew} of Assumption~\ref{th:main_assump} on the operator $G$ by an implicit one, namely extending item \ref{it:spaces} of Assumption~\ref{th:main_assump} by
%%%supposing that additionally the Banach space $\ix_1$ has also a separable pre-dual space $\iz_1$ with $\iz_1^*=\ix_1$. Then the Banach-Alaoglu theorem  also applies to $\ix_1$, and  sequences $\{u_n\} \subset \ix_1$ bounded in $\ix_1$ with weak$^*$-limit $\tilde u$ in $\ix$
%%possess subsequences $\{u_{n_k}\}$ convergent in $\ix_1$, where $\tilde u$ is again the limit element, which belongs to $\ix_1$. Here, the weak$^*$ lower semi-continuity of the norm functional in $\ix_1$ yields the sequentially weak$^*$ lower semi-continuity of the penalty functional $\Omega$.
%%\end{remark}
%
\section{Error estimate and a priori parameter choices}
\label{sec:apriori_parameter_choices}
We start with an error estimate result that provides the basis for the analysis of the regularizing properties, including convergence rates under a priori parameter choices. In what follows, we use the notation
\begin{align}
\mykap \defeq \frac{1}{\fab}.
\label{eq:mykapdef}
\end{align}
\begin{theorem}
\label{th:upardel-esti}
\mainassump
Then there exist finite positive constants $ K_1, \para_0 $ and $ \delta_0 $ such that for $ 0 < \para \le \para_0 $ and $ 0 < \delta \le \delta_0 $,
an error estimate for the regularized solutions as
\begin{align}
\label{eq:fin}
\norm{\upardel -\ust} \le f_1(\para) + K_1 \frac{\delta}{\para^{\mykap a}}
\end{align}
holds, where $f_1(\para)$ for $ 0 < \para \le \para_0 $ is some bounded function satisfying:
\begin{mylist_indent}
\item (No explicit smoothness)
If $ \ust \in \overline{\R(G)} $, then $ f_1(\para) \to 0 $ as $ \para \to 0 $.
\item (H\"older smoothness)
If $ \ust \in \ix_\pp $ for some $ 0 < \pp \le 1 $, then
$ f_1(\para) = \Landauno{\para^{\mykap p}} $ as $ \para \to 0 $.
\item (Low order smoothness)
If $ \ust \in \domain(\log \G) $, then
$ f_1(\para) = \Landauno{\loginv{\frac{1}{\para}}} $ as $ \para \to 0 $.
\end{mylist_indent}
\end{theorem}

Theorem \ref{th:upardel-esti}, the proof of which can be found in \cite{PlatoIndien22}, allows us to derive regularizing properties
of variational regularization with oversmoothing penalty and to obtain convergence and rates results for appropriate a priori parameter choices that
culminate in Theorem~\ref{th:apriori}.
For evaluating the strength of smoothness for the three different occurring situations in Theorem~\ref{th:upardel-esti}
(no explicit smoothness, H\"older smoothness and low order smoothness) we recall the chain \eqref{eq:chain2} of range conditions.

The following theorem is a direct consequence of Theorem~\ref{th:upardel-esti}, because its proof is immediately based on the error estimate \eqref{eq:fin} with the respective properties
of the function $f_1(\para)$.

\begin{theorem}
\label{th:apriori}
\mainassump
\begin{mylist}
\item (No explicit smoothness)
Let $ \ust \in \overline{\R(G)} $. Then for any a priori parameter choice
$\pardel=\para(\delta)$
satisfying
$ \pardel \to 0 $ and $ \tfrac{\delta}{\pardel^{\mykap a}} \to 0 $ as $ \delta \to 0 $,
we have
$$ \normix{u_{\pardel}^\delta-\ust} \to 0 \quad \textup{ as } \delta \to 0.
$$

\item (H\"older smoothness)
Let $ \ust \in \ix_p $ for some $ 0 < p \le 1 $.
Then for any a priori parameter choice satisfying
$ \para_*=\para(\delta) \sim  \delta^{1/(\kappa(p+a))} $,
we have
\begin{align*}
\normix{\udel -\ust} = \Landauno{\delta^{p/(p+a)}} \as \delta \to 0.
\end{align*}
\item (Low order smoothness)
Let $ \ust \in \domain(\log G) $.
Then for any a priori parameter choice satisfying
$ \para_*=\para(\delta) \sim \delta $, we have
\begin{align*}
\normix{\udel -\ust} = \Landauno{\loginv{\tfrac{1}{\delta}}} \as \delta \to 0.
\end{align*}
\end{mylist}
\end{theorem}

\section{Results for a discrepancy principle and consequences}
\label{sec:discrepancy_principle}
%
%%We start with the presentation of the discrepancy principle.
%
%%We start with an error estimate result that provides the basis for the analysis of the regularizing properties, including convergence rates under a priori parameter choices.
%
%
%
For the specification of a suitable discrepancy principle,
the behaviour of the misfit functional $ \para \mapsto \Jdel{\upardel} $ needs to be understood, for $ \delta > 0 $ fixed.
The basic properties are summarized in the following proposition.
As a preparation, we introduce the following parameter:
%% $\er $ is defined as follows:
%
\begin{align}
\er = \left\{\begin{array}{ll} 1,
& \textup{if } r \ge 1, \\
2^{-1+1/r}& \textup{otherwise}.
\end{array}\right.
\label{eq:er}
\end{align}
\begin{proposition}
\label{th:misfit-behavior}
\mainassump
Then for $ \delta > 0 $ fixed, the function \linebreak $ \para \mapsto \Jdel{\upardel} $ is \monotonically \nondecreasing, with
\begin{align*}
\lim_{\para \to 0}  \Jdel{\upardel} \le \er \delta,
\quad \lim_{\para \to \infty}  \Jdel{\upardel} =  \Jdel{\ubar}.
%%%\label{eq:misfit-behavior}
\end{align*}
In addition, we have $ \lim_{\para \to \infty} \normix{\upardel - \ubar} = 0 $.
\end{proposition}
\proof This follows along the lines of the proof of \cite[Proposition 4.5]{HofPla20}.
Details are thus omitted here.
\proofend
\begin{algorithm}[Discrepancy principle]
\label{th:discrepancy-def}
Let $ \myb > \er $ and $ \myc > 1 $ be finite constants.
\begin{myenumerate}
\item
If $ \Jdel{\ubar} \le \myb \delta $ holds,
then choose $ \pardel = \infty $, i.e., $ \upardel[\infty] \defeq \ubar \in \Dset $.

\item
Otherwise, choose a finite parameter $ \para =: \pardel > 0 $ such that
\begin{align}
\Jdel{\updd} \le \myb \delta \le \Jdel{\upddb}
\quad \textup{for some } \pardel \le \pardelb \le \myc \pardel,
\label{eq:parameter-choice}
\end{align}
where $ c > 1 $ denotes some finite constant.
\end{myenumerate}
\end{algorithm}%
\begin{remark}
%%\label{th:sdp}
\begin{myenumerate}
\item
Practically, a parameter $ \pardel $ satisfying condition \refeq{parameter-choice} can be determined, e.g., by a sequential discrepancy principle.
For more details, see, e.g.,
Section \ref{sec:study} and \cite{HofPla20}.

\item
It follows from Proposition \ref{th:misfit-behavior} that Algorithm
\ref{th:discrepancy-def} is feasible. We note that for $ \delta > 0 $ fixed, the function $ \para \mapsto \Jdel{\upardel} $ may be discontinuous. For this reason, we do not consider other versions of the discrepancy principle, e.g.,
$ \Jdel{\updd} = \myb \delta $ or
$ \myb_1 \delta \le \Jdel{\updd} \le \myb_2 \delta $.
\end{myenumerate}
\end{remark}

%%\noindent
We next present the main result of this paper.
%%The regularizing properties of the discrepancy principle are as follows.
%
\begin{theorem}
\label{th:discrepancy_main_result}
\mainassumpb, and let the regularization parameter $\para_*=\para(\delta,\fdelta)$ be chosen according to the discrepancy principle.
\begin{mylist}
\item (No explicit smoothness)
If $ \ust \in \overline{\R(G)} $, then
%%% for any a priori parameter choice
%%$\pardel=\para(\delta)$
%%%satisfying
%
we have
$$
\normix{u_{\pardel}^\delta-\ust} \to 0
\as \delta \to 0.
%%%\quad
%%\tfrac{\delta}{\pardel^{\mykap a}} \to 0 \as \delta \to 0.
$$

\item (H\"older smoothness)
If $ \ust \in \ix_p $ for some $ 0 < p \le 1 $,
then
%%for any a priori parameter choice satisfying
%%$ \para_*=\para(\delta) \sim  \delta^{1/(\kappa(p+a))} $
we have
\begin{align*}
\normix{\udel -\ust} = \Landauno{\delta^{p/(p+a)}}
%%%\qquad \para_*^{-1} = \Landauno{\delta^{-1/(\kappa(p+a))}}
\as \delta \to 0.
\end{align*}
\item (Low order smoothness)
If $ \ust \in \domain(\log G) $,
then
\begin{align*}
\normix{\udel -\ust} = \Landauno{\loginv{\tfrac{1}{\delta}}} \as \delta \to 0.
\end{align*}%%
\end{mylist}%%
\end{theorem}%%
%
%%\noindent
The proof of Theorem \ref{th:discrepancy_main_result} is given in Section
\ref{proof_main:theorem}.
\section{An illustrative example}
\label{sec:example}
The following example with specific Banach spaces and nonlinear forward operator is intended to illustrate the assumptions stated above and to indicate the specificity of occurring non-reflexive Banach spaces.
This example should show that the general mathematical framework developed in this paper is applicable. The considered basis space is the non-reflexive and non-separable space $ \ix = \linf $
with the essential supremum norm $ \norm{\cdot} = \normmax{\cdot} $ possessing a separable pre-dual space $\iz=L^1(0,1)$. The generator $G$  of the scale of Banach spaces is given by
\begin{align}
[Gu](x) = \int_0^x u(\xi) \, d\xi \qquad (0 \le x \le 1, \quad u \in \linf).
\label{eq:exopG}
\end{align}
Below we give some properties of $ G $:
\begin{mylist_indent}
\item
The operator $ G: \linf \to \linf $ is of \postype
with constant $ \kappa_* = 2 $,
see, e.g., \cite{Plato95}.

\item
$G$ is a compact operator, which possesses a compact pre-adjoint operator $ \tilde G: \lone \to \lone $, which is characterized by $$ [\tilde Gv](x) = \int_x^1 v(\xi) \, d\xi \qquad (0 \le x \le 1, \quad v \in \lone).$$

\item $G$ has a trivial nullspace and a non-dense range
$$ \R(G) = W^{1,\infty}_0(0,1) :=  \inset{ u \in W^{1,\infty}(0,1) : u(0) = 0 }, $$
with $$ \overline{\R(G)} =
C_0[0,1] := \inset{u \in C[0,1]:u(0) = 0 }.$$

\end{mylist_indent}

{\parindent0em As} a consequence of the last item we have that $\ix_1=W^{1,\infty}_0(0,1)$ with $\|u\|_1:=\|u^\prime\|_\infty$.
%% and associated separable pre-dual space $\iz_1=W^{1,1}_0(0,1)$.
%

With $\ix=\yps=L^\infty(0,1)$, the nonlinear forward operator of this example is
$ F: \linf \to \linf $ given by
\begin{align}
[F(u)](x) = \exp((Gu)(x)) \qquad (0 \le x\le 1,\quad u \in \linf).
\label{eq:exopF}
\end{align}
This operator $F$ is weak$^*$-to-weak sequentially continuous, because $F$ is a composition
of the continuous outer nonlinear exponential operator and the inner linear integration operator $G$, both mapping in $L^\infty(0,1)$. The operator $G$ transforms weak$^*$-convergent sequences in  $L^\infty(0,1)$ to norm-convergent sequences in this space, because $G$ is compact and has a pre-adjoint operator (cf.~\cite[Lemma~2.5]{Gatica18}).

Moreover, the operator $ F $ is Fr\'{e}chet differentiable on its domain of definition $ \domain(F) = \linf $, with
$ [F^\prime(u)] h=[F(u)] \cdot Gh $.
Now consider some function $ \ust \in \linf $ which is assumed to be fixed throughout this section. We then have
\begin{align*}
c_1 \le F\ust \le c_2 \ \textup{ on } [0,1], \; \mbox{with} \;
c_1 := \exp(-\normmax{G\ust}) > 0, \ c_2 \defeq \exp(\normmax{G\ust}),
\end{align*}
so that
\begin{align}
c_1 \vert Gh \vert \le  \vert F^\prime(\ust)h \vert \le c_2 \vert Gh \vert \quad \textup{on } [0,1] \qquad (h \in \linf).
\label{eq:Fprimebounds}
\end{align}
For any $ u \in \linf $, we denote by $ \Delta = \Delta(u) $ and
$ \theta = \theta(u) $ the following functions:
\begin{align*}
\Delta & \defeq Fu-F\ust \in \linf, \qquad \theta \defeq G(u-\ust) \in \linf.
\end{align*}
Thus,
$ \normmone{u-\ust} = \normmax{\theta} $, and we refer to \refeq{taunorm} for the definition of
$ \normmone{\cdot}$.

Below we show that the basic estimates \refeq{normequiv_a} and
\refeq{normequiv_b} are satisfied for that example with $a=1$. As a preparation, we note that
\begin{align}
\vert \Delta-  F^\prime(\ust) (u-\ust) \vert
\le \vert \theta \vert \, \vert \Delta \vert \quad \textup{on } [0,1],
\label{eq:hofmann-estimate}
\end{align}
and refer in this context to \cite[Sect.~4.4]{HHMP22}. In this reference,
the  same $F$ is analyzed as an operator mapping in $L^2(0,1)$, where moreover
its relation to a parameter estimation problem
for an initial value problem of a first order ordinary differential equation
is outlined.
\begin{myenumerate}
\item
We first show that \refeq{normequiv_a} holds. Even more general we show that it holds for any $ u \in \linf $ sufficiently close to $ \ust $, not only for $ u \in \ix_1$.
From \refeq{Fprimebounds} we have that
\begin{align*}
\vert \Delta-  F^\prime(\ust) (u-\ust) \vert
\ge
\vert \Delta \vert -  \vert F^\prime(\ust) (u-\ust) \vert
\ge
\vert \Delta \vert - c_2 \vert \theta \vert \quad \textup{on } [0,1],
\end{align*}
and \refeq{hofmann-estimate} then implies the estimate
\begin{align*}
\vert \Delta \vert - c_2 \vert \theta \vert
\le
\vert \theta \vert \ \vert \Delta \vert
\quad \textup{on } [0,1].
\end{align*}
For any $ u \in \linf $ satisfying $ \normmax{\theta} \le \myrr < 1 $, we thus have
$ \vert \Delta \vert \le \myrr \vert \Delta \vert + c_2 \vert \theta \vert$
and therefore
$ (1-\myrr)\vert \Delta \vert \le c_2 \vert \theta \vert $
on $ [0,1] $.
This finally yields
$$ \tfrac{1-\myrr}{c_2} \normmax{\Delta} \le \normmax{\theta} \quad \textup{for }
\normmax{\theta} \le \myrr \qquad (0 < \myrr < 1),
$$
from which the first required nonlinearity condition \refeq{normequiv_a} follows immediately.

\item
We next show that \refeq{normequiv_b} holds, in fact for any $ u \in \linf $  sufficiently close to $ \ust $.
From \refeq{Fprimebounds} we have
\begin{align*}
\vert \Delta-  F^\prime(\ust) (u-\ust) \vert
\ge
\vert F^\prime(\ust) (u-\ust) \vert - \vert \Delta \vert
\ge
c_1 \vert \theta \vert - \vert \Delta \vert \quad \textup{on } [0,1],
\end{align*}
and \refeq{hofmann-estimate} then implies that
\begin{align*}
c_1 \vert \theta \vert \le \vert \Delta \vert
+ \vert \theta \vert \, \vert \Delta \vert
\quad \textup{on } [0,1].
\end{align*}
For any $ 0 < \varepsilon < c_1 $ and $ u \in \linf $ satisfying $ \normmax{\Delta} \le c_1 - \varepsilon $, we thus have
$ c_1\vert \theta \vert \le \vert \Delta \vert + (c_1-\varepsilon) \vert \theta \vert
$
and therefore $ \varepsilon \vert \theta \vert \le \vert \Delta \vert $
on $ [0,1] $.
This provides us with the estimate $ \varepsilon \normmax{\theta} \le \normmax{\Delta},$ which is valid for
$\normmax{\Delta} \le c_1-\varepsilon \; (0 < \varepsilon < c_1).$
This, however, yields directly the second required nonlinearity condition \refeq{normequiv_b} and completes the list of requirements imposed by Assumption~\ref{th:main_assump}.
\end{myenumerate}
\section{Numerical case studies} \label{sec:study}
In this section, we verify the main result in Theorem \ref{th:discrepancy_main_result} for the situation of H\"older smoothness.
For this, we recall the example considered in Section \ref{sec:example}.
In particular, for the spaces $\ix =\mathcal{Y}=\linf$, equipped with the essential supremum norm $ \norm{\cdot} = \normmax{\cdot} $, we consider the operator equation (\ref{eq:opeq}), where the nonlinear forward operator $ F: \linf \to \linf $ is given by (\ref{eq:exopF}).
In this context, the integration operator $G$, defined as in (\ref{eq:exopG}),
generates the space
$$\ix_{1}=\R(G)=W^{1,\infty}_0(0,1) :=  \inset{ u \in W^{1,\infty}(0,1) : u(0) = 0 }\,,$$
with norm $\|u\|_1:=\|u^\prime\|_\infty$.
As verified in the previous section, this example satisfies Assumption \ref{th:main_assump} with $a=1$.

In the numerical experiments presented below, we consider the model equation $ F(u) = \fst $ with $ \fst(x)=\exp(x^{p+1}/(p+1)), \, 0 \le x \le 1 $,
with two different values $p$ from the interval $(0,1)$.
The solution is obviously given by  $u^{\dagger}(x)=x^{p}$ for $ 0 \le x \le 1 $.
It satisfies $u^{\dagger}\in \ix_{p}$ which
follows from the fact that the fractional powers $ G^p $ coincide with
Abel integral operators, and thus $ [G^p \Gamma(1+p)](x) = x^p, \, 0 \le x \le 1 $,
where $ \Gamma $ denotes Euler's gamma function.
For details, we refer to
\cite[p.~9]{Gorenflo_Vessella91} and \cite{Plato97.2}.
Note that $u^{\dagger} \notin \ix_{1}$,
hence we have an oversmoothing penalty term in the Tikhonov functional $\pad{\myu}$ defined as in (\ref{eq:TR}).
To find a regularized solution for $u^{\dagger}$, we set $\myubar=0$ as an initial guess and $r=1$ in the minimization problem (\ref{eq:TR}).
We use R programming software \cite{RProg} for the implementation.
The interval $[0,1]$ is partitioned by using equidistant grid points $0=x_0<\ldots<x_N=1$ with $N=100$.
To approximate the functions $u$ on the given grid,
we exploit linear splines
that vanish at $ x = 0 $.
In what follows, we use the notation $$\lVert u\rVert:=\max_{i=0,\ldots,N}|u_{i}|$$ for the discrete norm.
We simulate perturbed observations $\fdelta_{i}$, $i=0,\ldots,N$, as follows:
\[
\fdelta_{i}=
\begin{cases}
f^{\dagger}_{i}+\delta\frac{\rho_{i}}{\lVert \rho \rVert}, \quad &i=1,\ldots,N\,, \\
f^{\dagger}_{0} \quad &i=0\,.
\end{cases}\]
In this setting, $f^{\dagger}_{i}=(F(u^{\dagger}))_{i}=\exp(x_{i}^{p+1}/(p+1))$, for $i=0,\ldots,N$,
denotes the simulated right-hand side of operator equation (\ref{eq:opeq}), and the vector $\rho=(\rho_{1},\ldots,\rho_{N})^{T}$ consists of independent and identically distributed standard Gaussian variables $\rho_{i}$, for $i=1,\ldots,N$.
The discrepancy principle is implemented sequentially as follows (see Remark 4.8 in \cite{HofPla20}):
\begin{itemize}
\item Choose initial constants $b>e_{r}$, $\theta>1$, and $\alpha^{(0)}>0$.
\item If $\lVert F(u_{\alpha^{(0)}}^{\delta})-{\fdelta}\rVert\geq b\delta$ holds, proceed with $\alpha^{(k)}=\theta^{-k}\alpha^{(0)}$, for $k=1,2\ldots$, until $\lVert F(u_{\alpha^{(k)}}^{\delta})-{\fdelta}\rVert\leq b\delta\leq \lVert F(u_{\alpha^{(k-1)}}^{\delta})-{\fdelta}\rVert$ is satisfied for the first time.
On that occasion, set $\alpha_{\ast}=\alpha^{(k)}$.
\item If $\lVert F(u_{\alpha^{(0)}}^{\delta})-f^{\delta}\rVert\leq b\delta$ holds, proceed with $\alpha^{(k)}=\theta^{k}\alpha^{(0)}$, for $k=1,2\ldots$, until $\lVert F(u_{\alpha^{(k-1)}}^{\delta})-\fdelta\rVert\leq b\delta\leq \lVert F(u_{\alpha^{(k)}}^{\delta})-{\fdelta}\rVert$ is satisfied for the first time.
Set $\alpha_{\ast}=\alpha^{(k-1)}$.
\end{itemize}

Within the minimization steps, we use the command \texttt{fminunc} included in the package \texttt{pracma}.
Table \ref{tablep03} and \ref{tablep07} illustrate the results of Algorithm \ref{th:discrepancy-def} for $p=0.3$ and $p=0.7$, respectively, and for decreasing values of $\delta$.
The initial values are chosen as $b=2$, $\theta=2$, and $\alpha^{(0)}=1$. Except for the second column of Table \ref{tablep03}, all values are rounded to four decimal places.
The second columns of the tables present the values of the regularization parameter $\alpha_{\ast}$ chosen by the discrepancy principle.
The third columns illustrate the corresponding regularization errors. The last columns confirm the statement of Theorem \ref{th:discrepancy_main_result}.
\begin{table}
\caption{Numerical results of Algorithm \ref{th:discrepancy-def} for $p=0.3$.}\label{tablep03}
\begin{center}
\begin{tabular}{@{}ccccc}
\hline
$\mathbf{\delta}$  & $\alpha_{\ast}$ &$\lVert u_{\alpha_{\ast}}^{\delta}-u^{\dagger}\rVert$ & $\dfrac{\lVert u_{\alpha_{\ast}}^{\delta}-u^{\dagger}\rVert}{\delta^{p/(p+1)}}$ \\
\hline
0.0500 & $7.813 \cdot 10^{-3}$ & 0.2118  & 0.4228 \\
0.0250 & $7.813 \cdot 10^{-3}$ & 0.2124  & 0.4975 \\
0.0125 & $1.953 \cdot 10^{-3}$ & 0.1929  & 0.5304 \\
0.0062 & $2.441 \cdot 10^{-4}$ & 0.1496  & 0.4827 \\
0.0031 & $4.883 \cdot 10^{-4}$ & 0.1559  & 0.5902 \\
0.0016 & $1.221 \cdot 10^{-4}$ & 0.1217  & 0.5405 \\
0.0008 & $6.104 \cdot 10^{-5}$ & 0.1157  & 0.6033 \\
0.0004 & $3.052 \cdot 10^{-5}$ & 0.0920  & 0.5625 \\
0.0002 & $7.629 \cdot 10^{-6}$ & 0.0919  & 0.6597 \\
\hline
\end{tabular}
\end{center}
\end{table}
\begin{table}
\caption{Numerical results of Algorithm \ref{th:discrepancy-def} for $p=0.7$.}\label{tablep07}
\begin{center}
\begin{tabular}{@{}cccc}
\hline
$\mathbf{\delta}$  & $\alpha_{\ast}$ &$\lVert u_{\alpha_{\ast}}^{\delta}-u^{\dagger}\rVert$  & $\dfrac{\lVert u_{\alpha_{\ast}}^{\delta}-u^{\dagger}\rVert}{\delta^{p/(p+1)}}$ \\
\hline
  0.0500 & 0.0156 & 0.0916  & 0.3146 \\
  0.0250 & 0.0156 & 0.0869  & 0.3967 \\
  0.0125 & 0.0156 & 0.0703  & 0.4273 \\
  0.0062 & 0.0078 & 0.0444  & 0.3586 \\
  0.0031 & 0.0078 & 0.0532  & 0.5726 \\
  0.0016 & 0.0078 & 0.0253  & 0.3619 \\
  0.0008 & 0.0005 & 0.0246  & 0.4682 \\
  0.0004 & 0.0010 & 0.0171  & 0.4323 \\
  0.0002 & 0.0001 & 0.0125  & 0.4209 \\
\hline
\end{tabular}
\end{center}
\end{table}

Figure \ref{FigureA} shows the shapes (solid lines) of minimizers $u_{\alpha}^{\delta}$ of the Tikhonov functional in the case $p=0.3$ for a fixed noise level $\delta=0.0125$ and a series of regularization parameters $\alpha>0$ with decreasing values.
The relative error $\lVert F(u^{\dagger})-\fdelta\rVert/{\lVert F(u^{\dagger})\rVert}$ is given by $0.01$.
Dotted lines represent in all five pictures the graph of the solution $u^{\dagger}(x)=x^{0.3}, \, 0 \le x \le 1$, to be reconstructed. In the first picture, for the largest $\alpha$, the regularized solution is too smooth. By reducing the values of $\alpha$, the recovery gets
improved. Precisely, the third picture yields the best approximate solution, which corresponds with $\alpha_*$ from the discrepancy principle. As a consequence of the ill-posedness of the problem, more and more oscillating solutions occur when $\alpha$ further tends to zero.

%%\begin{figure}
%%\includegraphics[trim=6cm 7.5cm 0  5cm, scale=1,page=1]{HKP_Graphs.pdf}
%%%\caption{Behaviour of minimizing functions $u_{\alpha}^{\delta}$ for $\delta=0.0125$ and decreasing values of $\alpha$.\label{FigureA}}
%%%\end{figure}

\begin{figure}
\caption{Behaviour of minimizing functions $u_{\alpha}^{\delta}$ for $\delta=0.0125$ and decreasing values of $\alpha$.\label{FigureA}}
\begin{center}
\centering
 \begin{tikzpicture}[trim axis left, scale=0.8]

    \begin{axis}[ name=plot1,
    xmin = 0, xmax = 1,
    ymin = 0, ymax = 1.3,
     xlabel=$x$, ylabel=$u(x)$,
     at={(-3.66,0)}
   ]
        \addplot[black,
         mark = .,
         thick] table [x = {"x"}, y = {"y1"}] {\mytable};
        \addplot[black,
        mark=.,
        thick,
        domain=0:1,
        smooth,
        samples=100,
        dotted]{x^0.3};
        \node[draw] at (0.7,0.2) {\begin{tabular}{c} $\alpha=1.25 \cdot 10^{-1}$ \\ $\lVert u_{\alpha}^{\delta}-u^{\dagger}\rVert=0.418$ \end{tabular}};
    \end{axis}

    \begin{axis}[name=plot2,
    xmin = 0, xmax = 1,
    ymin = 0, ymax = 1.3,
    xlabel=$x$, ylabel=$u(x)$,
    at=(plot1.right of south east),
     anchor=left of south west]
        \addplot[black, mark = .,thick] table [x = {"x"}, y = {"y2"}] {\mytable};
        \addplot[black,
        mark=.,
        thick,
        domain=0:1,
        smooth,
        samples=100,
        dotted]{x^0.3};
        \node[draw] at (0.7,0.2) {\begin{tabular}{c} $\alpha=1.56\cdot10^{-2}$ \\ $\lVert u_{\alpha}^{\delta}-u^{\dagger}\rVert=0.222$ \end{tabular}};
    \end{axis}

%%\newpage
%%\end{tikzpicture}
%%\end{center}
%%%\end{figure}
%%
%%\begin{figure}
%%\caption{Behaviour of minimizing functions $u_{\alpha}^{\delta}$ for $\delta=0.0125$ and decreasing values of $\alpha$.\label{Figure1a}}
%%\begin{center}
%%% \begin{tikzpicture}[trim axis left, scale=0.8]
    \begin{axis}[name=plot4,
    xmin = 0, xmax = 1,
    ymin = 0, ymax = 1.3,
    xlabel=$x$, ylabel=$u(x)$,
    at=(plot2.below south west), anchor=above north west]
        \addplot[black, mark = .,thick] table [x = {"x"}, y = {"y4"}] {\mytable};
        \addplot[black,
        mark=.,
        thick,
        domain=0:1,
        smooth,
        samples=100,
        dotted]{x^0.3};
        \node[draw] at (0.7,0.2) {\begin{tabular}{c} $\alpha=2.44\cdot10^{-4}$ \\ $\lVert u_{\alpha}^{\delta}-u^{\dagger}\rVert=0.219$ \end{tabular}};
    \end{axis}

    \begin{axis}[name=plot3,
    xmin = 0, xmax = 1,
    ymin = 0, ymax = 1.3,
    xlabel=$x$, ylabel=$u(x)$,
    at=(plot4.left of south west), anchor=right of south east]
        \addplot[black, mark = .,thick] table [x = {"x"}, y = {"y3"}] {\mytable};
        \addplot[black,
        mark=.,
        thick,
        domain=0:1,
        smooth,
        samples=100,
        dotted]{x^0.3};
        \node[draw] at (0.7,0.2) {\begin{tabular}{c} $\alpha=1.95\cdot10^{-3}$ \\ $\lVert u_{\alpha}^{\delta}-u^{\dagger}\rVert=0.193$ \end{tabular}};
    \end{axis}

    \begin{axis}[name=plot6,
    xmin = 0, xmax = 1,
    ymin = 0, ymax = 1.3,
    xlabel=$x$, ylabel=$u(x)$,
    at=(plot3.below south west), anchor=above north west]
        \addplot[black, mark = .,thick] table [x = {"x"}, y = {"y6"}] {\mytable};
        \addplot[black,
        mark=.,
        thick,
        domain=0:1,
        smooth,
        samples=100,
        dotted]{x^0.3};
        \node[draw] at (0.7,0.2) {\begin{tabular}{c} $\alpha=1.53\cdot10^{-5}$ \\ $\lVert u_{\alpha}^{\delta}-u^{\dagger}\rVert=0.465$ \end{tabular}};
    \end{axis}

\end{tikzpicture}
\end{center}
\end{figure}
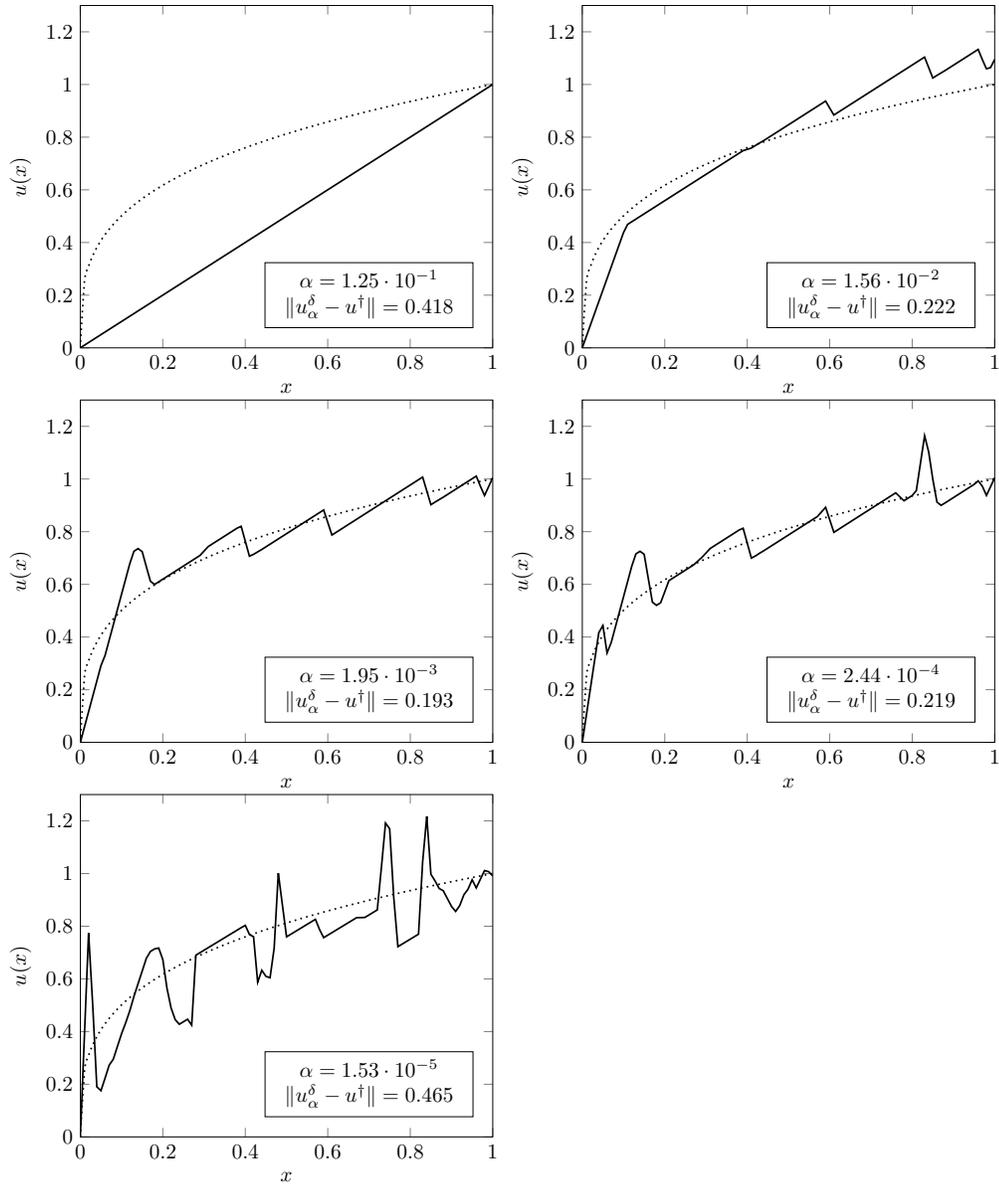
\section{Constructions and verifications}
\label{sec:verifications}
In this section, we verify the main result of the paper. For this purpose, we return to   the general setting considered in Section \ref{sec:preparations}, i.e., $ G: \ix \to \ix  $ denotes a bounded linear operator which is of \postype, one-to-one and has an unbounded inverse, where $ \ix $ is a Banach space.
\subsection{Introduction of auxiliary elements}
For the auxiliary elements introduced below, we consider linear bounded regularization operators associated with $ \G $,
\begin{align}
R_\parb: \ix \to \ix	\for \parb > 0
\label{eq:rbeta}
\end{align}
and its companion operators
\begin{align}
S_\parb \defeq  I - R_\parb \G \for \parb > 0.
\label{eq:sbeta}
\end{align}
Throughout this section, we assume that the following conditions are satisfied:
\begin{align}
\norm{ R_\parb }_{\scriptscriptstyle \mathcal{L}(\ix)} & \le \tfrac{c_*}{\parb} \for \parb > 0,
\label{eq:wachstum} \\
\norm{ S_\parb G^\pp }_{\scriptscriptstyle \mathcal{L}(\ix)} & \le \cp  \parb^{\pp} \for \parb> 0,
\qquad (0 \le \pp \le \pp_0)
\label{eq:abfall} \\
R_\parb G & = G R_\parb \for \parb > 0,
\label{eq:commute}
\end{align}
where $ 0 < \pp_0 < \infty $ is a finite number to be specified later, and $ c_* $ and $ \cp $ denote finite constants. We assume that $ \cp $ is bounded as a function of $ p $.
\begin{example} An example is given by Lavrentiev's $m$-times iterated method with an integer $ m \ge 1 $. Here, for $ f \in \ix $ and $ v_0 = 0 \in \ix $, the element $ R_\parb f $ is given by
\begin{align*}
(\G + \parb I)v_n & =  \parb v_{n-1} + f \for n = 1,2,\ldots,m, \qquad
R_\parb f \defeq v_m.
\end{align*}
The operator $ R_\parb $ can be written in the form
$$ R_\parb = \parb^{-1} \sum_{j=1}^{m} \parb^j (G + \parb I)^{-j},  $$
and the
companion operator is given by $ S_\parb = \parb^m (G + \parb I)^{-m} $.
For $ m = 1 $, this gives Lavrentiev's classical regularization method,
$ R_\parb = (G + \parb I)^{-1} $.
For this method, the conditions \refeq{wachstum}--\refeq{commute} are satisfied with
$ p_0 = m $. In fact,
for integer $ 0 \le p \le m $,
estimate \refeq{abfall} holds with constant $ c_p = (\kappa_*+1)^m $, see
\cite[Lemma 1.1.8]{Plato95}.
From this intermediate result and the interpolation inequality
\refeq{interpol2}, inequality \refeq{abfall} then follows
for non-integer values $ 0 < p < m $, with constant $ c_p = 2(\kappa_*+1)^{m+1} $.
\remarkend
\end{example}
We are now in a position to introduce \emph{auxiliary elements}
which provide an essential tool for the analysis of the regularization properties of \tikreg considered in our setting.
They are defined as follows,
\begin{align}
\ualpaux \defeq \ubar + R_\parb \G (\ust-\ubar)
=
\ust - S_\parb(\ust-\ubar)
\for \parb > 0,
\label{eq:uaux-def}
\end{align}
where $\G$ is the generator of the scale of normed spaces introduced in Section \ref{postype_operators}, and $ R_\parb, \parb > 0 $, is an arbitrary family of regularizing operators as in \refeq{rbeta} satisfying the conditions \refeq{wachstum}--\refeq{commute}
with saturation
$$ p_0 \ge 1+a, $$
and $ S_\parb, \parb > 0 $, denotes the corresponding companion operators, cf.~\refeq{sbeta}.
In addition,
the solution $ \ust$ of the operator equation \eqref{eq:opeq} and the corresponding initial guess $\ubar $ are as introduced above.
The basic properties of the auxiliary elements \refeq{uaux-def} are summarized in Lemma~\ref{th:auxel_1} below.

We next state another property of regularization operators which is also needed below.
\begin{lemma}
\label{th:auxel_0}
%%Let $ R_\parb, \parb > 0 $, be an arbitrary family of regularizing operators as in \refeq{rbeta} satisfying the conditions \refeq{wachstum}--\refeq{commute}.
%%Then
There exists some finite constant $ c > 0  $ such that, for each $ 0 < \pp \le 1$, we have
\begin{align*}
\norm{ R_\parb \G^\pp }_{\scriptscriptstyle \mathcal{L}(\ix)} & \le c \parb^{\pp-1} \for \parb > 0.
\end{align*}
\end{lemma}
\begin{proof}
Since $ R_\parb \G^\pp = \G^\pp R_\parb $, for
$ \kappa_1 = 2(\kappa_*+1) $
we have
\begin{align*}
\norm{ R_\parb \G^\pp w } & =
\norm{ G^\pp R_\parb w } \le \kappa_1 \norm{ G R_\parb w }^\pp \norm{ R_\parb w }^{1-\pp}
\\
& \le
\kappa_1(c_0+1)^p c_*^{1-p} \norm{ w } \parb^{\pp-1}, \qquad w \in \ix,
\end{align*}
where the first inequality follows from the interpolation inequality \refeq{interpol2}.
For the meaning of the constants $ c_0 $ and $ c_* $, we refer to
\refeq{wachstum} and \refeq{abfall}, respectively.
\end{proof}
\subsection{Auxiliary results for $ \log \G $}
\begin{lemma}
\label{th:low-order-rate}
For each $ \myu \in \domain(\log \G) $ and each $ 0 \le p < \pp_0 $, we have
$$
\norm{ S_\parb G^p \myu  } =
\Landauno{\beta^p \loginv{\tfrac{1}{\parb}}}
\as \beta \to 0.
$$
\end{lemma}
\begin{proof}
By $ (\G^q)_{q \ge 0} $ a $ C_0 $-semigroup on $\ixtwo $ is defined, and thus
$ \norm{G^q}_{\scriptscriptstyle \mathcal{L}(\ixtwo)} \le C e^{\omega q} $ for $ q \ge 0 $,
where $ \omega > 0 $ and $ C > 0 $ denote suitable constants, and $ \norm{\cdot}_{\scriptscriptstyle \mathcal{L}(\ixtwo)} $ denotes the norm of operators on $ \ixtwo $.
%%%This follows, e.g., from the fact
Therefore, each real $ \lambda > \omega $ belongs to the resolvent set of the operator $ \log \G: \RG \supset \domain(\log \G) \to \RG $, i.e.,
$ (\lambda I - \log \G)^{-1}: \RG \to \RG $ exists and defines a bounded operator,
cf.~\cite[Theorem 5.3, Chapter~1]{Pazy83}.
Since
$$ \R((\lambda I - \log \G)^{-1}) = \domain (\lambda I - \log \G) = \domain( \log \G), $$
we can represent $ u $ as
$$ u = (\lambda I - \log \G)^{-1}w $$
with some $ w \in \RG $. From (cf.~\cite[proof of Theorem 5.3, Chapter 1]{Pazy83})
$$  u = (\lambda I - \log \G)^{-1}w = \int_0^\infty e^{-\lambda q} G^q w \, dq, $$
we obtain
\begin{align*}
S_\parb G^p u = \int_0^\infty e^{-\lambda q} S_\parb G^{p+q} w \, dq = y_1 + y_2,
\end{align*}
with
\begin{align*}
y_1 = \int_0^{\pp_0-p} e^{-\lambda q} S_\parb G^{p+q} w \, dq,
\qquad y_2 = \int_{\pp_0-p}^\infty e^{-\lambda q} S_\parb G^{p+q} w \, dq.
\end{align*}
%
%%which provides the fundamental identity for the
Below we provide suitable estimates for $ y_1 $ and $ y_2 $.
The former term can be estimated as follows for $ \beta < 1 $:
\begin{align*}
\normix{y_1} & \le  c_1 \norm{w} \int_0^{\pp_0-p}  \parb^{p+q}\, dq =
c_1 \norm{w} \parb^p \frac{1}{\log \parb} \parb^q \big\vert_{q=0}^{q=\pp_0-p}
\\
& = c_1 \norm{w}  \parb^p \frac{1}{\vert \log \parb \vert} (1-\parb^{\pp_0-p})
\le c_1 \norm{w}  \parb^p \frac{1}{\vert \log \parb \vert},
\end{align*}
where $ c_1 $ denotes a finite constant.
The element $ y_2 $ can be written as follows:
\begin{align*}
y_2 = \int_{\pp_0-p}^\infty e^{-\lambda q} S_\parb G^{\pp_0} G^{q-(\pp_0-p)} w \, dq,
\end{align*}
and thus we can estimate as follows:
\begin{align*}
\normix{y_2} & \le  c_2 \norm{w} \int_{\pp_0-p}^\infty e^{-\lambda q} \parb^{\pp_0} e^{\omega(q-(\pp_0-p))} \, dq
\\
& \le c_3 \norm{w} e^{-\omega(\pp_0-p)} \parb^{\pp_0}  \int_{\pp_0-p}^\infty e^{-(\lambda-\omega)q} \, dq
=\Landauno{ \parb^{\pp_0}} \as \beta \to 0,
\end{align*}
where $ c_2 $ and $ c_3 $ denote suitable finite constants.
This completes the proof.
\end{proof}
\begin{lemma}
\label{th:low-order-rate-b}
For each $ \myu \in \domain(\log \G) $, we have
$$
\norm{ R_\parb \myu  } =
\Landauno{\frac{1}{\beta \log \frac{1}{\parb}}}
\as \beta \to 0.
$$
\end{lemma}
\begin{proof}
Follows similar to Lemma \ref{th:low-order-rate}, by making use of
Lemma \ref{th:auxel_0}. Details are thus omitted.
\end{proof}
\subsection{Some preparations for low order rates}
\label{low_order_preparations}
In the analysis of low order rates, the functions
\begin{align}
\myphi(t) & = \kla{\mylogb{t}}^{-1}, \quad 0 < t < 1,
\label{eq:phi-def} \\
\mychi{\pm 1}{q}(t) & = t^q \kla{\mylogb{t}}^{\mp 1}, \quad 0 < t < 1
%%\qquad \kla{p = \pm 1, q > 0},
\qquad \kla{q > 0},
\label{eq:chi-def}
\end{align}
will be needed. Below we state some elementary properties of those functions.
%%% $\mychi{p}{q} $ for the case $ p > 0 $
%%which turn out to be helpful in the further analysis.
Note that $ \myphi(t) = \mychi{1}{0}(t) $ holds, so
\refeq{phi-def} is a special case of \refeq{chi-def}.
%
%%\begin{mylist_indent}
\begin{myenumerate_indent}
\item
\label{it:low_oder_prep_a}
For $ q \ge 0 $,
the function $ \mychi{1}{q} $ is monotonically increasing on the interval $ \kla{0,1}  $, with $ \mychi{1}{q}(t) \to 0 $ as $ t \to 0 $.

\item
\label{it:low_oder_prep_b}
For $ q > 0 $, the function $ \mychi{-1}{q} $ is monotonically increasing on the interval $ (0,t_0] $,
with $ t_0 = t_0(q) < 1 $ chosen sufficiently small.
We have $ \mychi{-1}{q}(t) \to 0 $ as $ t \to 0 $.

\item
\label{it:low_oder_prep_c}
For $ q > 0 $, the inverse function $ \mychiinv{1}{q}: (0,1) \to \reza $
%%%, defined on an interval $ (0,s_0) $ for $ 0 < s_0 < 1 $ small enough,
satisfies
\begin{align*}
\mychiinv{1}{q}(s) \sim
%%q^{-p/q} \mychi{-\frac{p}{q}}{\frac{1}{q}}(s) \as s \to 0.
q^{-1/q} \mychi{-1}{1}(s)^{1/q} \as s \to 0.
\end{align*}
This in particular implies that, for each $ 0 < s_0 < 1 $, we have
\begin{align*}
%%\mychiinv{p}{q}(s) \asymp \mychi{-\frac{p}{q}}{\frac{1}{q}}(s),
\mychiinv{1}{q}(s) \asymp \mychi{-1}{1}(s)^{1/q},
\quad 0 < s \le s_0.
\end{align*}

\item
\label{it:low_oder_prep_d}
For each $q > 0 $, we have
$ \myphi(\mychi{\pm 1}{q}(t)) \sim q \myphi(t) $ as $ t \to 0 $. Thus, in particular
for each fixed $ t_1 $ small enough and each $ e > 0 $, we have
\begin{align*}
\myphi(\mychi{\pm 1}{q}(t)^e) \le c_1 \myphi(t), \quad 0 < t \le t_1,
%%%\label{eq:phi-prop1}
\end{align*}
for some appropriate constant $ c_1 $.

% yxxx

\item
\label{it:low_oder_prep_e}
For each constant $ c_2 > 0 $, we have
$ \myphi(c_2t) \sim \myphi(t) $ as $ t \to 0 $,
and thus in particular
\begin{align*}
\myphi(c_2t) \asymp \myphi(t), \quad 0 < t \le t_2,
%%\label{eq:phi-prop2}
\end{align*}
for $ t_2 < 1/c_2 $ fixed.
%% and some appropriate constant $ c_3 $.
%%\end{mylist_indent}
\end{myenumerate_indent}
Here, for two positive, real-valued functions $ f, g: (0,t_0) \to \reza $, the notation $ f(t) \sim g(t) $ as $ t \to 0 $ means $ f(t)/g(t) \to 1 $ as $ t \to 0 $. In addition, $ f(t) \asymp g(t) $ for $ t \in I \subset (0,t_0) $ means that there are finite positive constants $ c_1, c_2 $ such that
$ c_1 f(t) \le g(t) \le c_2 f(t) $ for $ t \in I $.
\subsection{Properties of auxiliary elements}
\label{auxels}
In this section, we present the basic properties of the auxiliary elements, which may be used to verify our convergence results presented below.
\begin{lemma}
\label{th:auxel_1}
Consider the auxiliary elements from \refeq{uaux-def},
generated by regularization operators $ R_\beta, \beta > 0 $, with saturation
$ \pp_0 \ge 1 + a $.
Let the three functions $g_i(\parb)\;(i=1,2,3)$ be given by the following identities:
\begin{align}
& \normix{\ualpaux - \ust}=g_1(\parb),
\label{eq:auxel-a}
\\
& \norma{\ualpaux - \ust} = g_2(\parb)\parb^{a},
\label{eq:auxel-b}
\\
& \normone{\ualpaux-\ubar} = g_3(\parb)\parb^{-1},
\label{eq:auxel-c}
\end{align}
for $ \parb > 0 $, respectively. Those functions $g_i(\parb)\;(i=1,2,3)$  are bounded and have the following properties:
\begin{mylist_indent}
\item (No explicit smoothness) If
$ \ust \in \overline{\R(G)} $, then we have $ g_i(\parb) \to 0 $ as $ \parb \to 0 $ ($i=1,2,3$).
\item (H\"older smoothness)
If $ \ust \in \ix_\pp $ for some $ 0 < \pp \le 1 $,
then
$ g_i(\parb)=\mathcal{O}(\parb^\pp) $ as $ \parb \to 0 $ ($i=1,2,3$).
\item (Low order smoothness)
If $ \ust \in \domain(\log \G) $,
then
$ g_i(\parb)= \Landauno{\loginv{\frac{1}{\parb}}} $
as $ \parb \to 0 $ ($i=1,2,3$).
\end{mylist_indent}
\end{lemma}
\proof
By definition, those three functions $ g_1, g_2 $ and $ g_3 $ under consideration can
be written as follows:
\begin{align*}
g_1(\parb) &= \| S_\parb (\ust-\ubar)\|,
\\
g_2(\parb) &=
\parb^{-a} \| G^a S_\parb (\ust-\ubar) \|, \\
g_3(\parb) &=
\parb \| R_\parb (\ust - \ubar) \|,
\end{align*}
and, according to conditions \refeq{wachstum}--\refeq{commute},
thus are bounded.
\begin{mylist}
\item
The three convergence statements under any missing smoothness assumptions all are verified by making use of the uniform boundedness principle.
We give some details for the function $ g_1 $.
%% by taking into account formula
%%\refeq{abfall} and
%%Lemma \ref{th:auxel_0}. To apply this principle, we consider the
%%parametric family of linear operators $S_\beta: \ix \to \ix$ and
%%in this context the associated limiting process $\beta \to 0$ of the
%%%parameter.
In fact, \refeq{abfall} applied for $ p = 1 $ gives
$S_\beta z \to 0$ as $\beta \to 0$ for all $z$ from the range
$\mathcal{R}(G)$.
The uniform boundedness $\|S_\beta\|_{\scriptscriptstyle \mathcal{L}(\ix)} \le c_0$, cf.~\refeq{abfall} for $ p = 0 $,
and the denseness of $ \R(G) $ in  $\overline{\R(G)}$ then gives
$g_1(\beta)=\|S_\beta (u^\dagger-\ubar )\| \to 0$ as $\beta \to
0$.
The assertions for $g_2$ and $g_3$ follow similarly.

\item We consider H\"older smoothness next. Since $ \ust, \, \ubar \in \ix_\pp $ holds, we have $ \ust - \ubar = G^{\pp}w $ for some $ w \in \ix $.
The statements are now easily obtained from \refeq{abfall} and
Lemma~\ref{th:auxel_0}.

\item
We have $ \ust - \ubar \in \domain(\log G) $, and the statements now easily follow from
Lemmas~\ref{th:low-order-rate} and \ref{th:low-order-rate-b}.

\proofend
\end{mylist}
The preceding lemma allows the construction of smooth approximations in $ \ix_1 $ to $ \ust $, which may be used in the subsequent proofs.
\begin{lemma}
\label{th:auxel_2}
Under the conditions of Lemma \ref{th:auxel_1}, the following holds:
\begin{mylist_indent}
\item (No explicit smoothness) If
$ \ust \in \overline{\R(G)} $, then
for some parameter choice $ \beta = \betast $
%% with $\betast \to \delta $,
we have
\begin{align}
\normix{\ualpauxst - \ust} \to 0,
\quad
\norma{\ualpauxst - \ust} = \Landauno{\delta},
\quad
\normone{\ualpauxst - \ubar} = \landauno{\delta^{-\lfrac{1}{a}}},
\label{eq:auxel-2-a}
\end{align}
as $ \delta \to 0 $.
%%%, respectively.

\item (H\"older smoothness)
If $ \ust \in \ix_\pp $ for some $ 0 < \pp \le 1 $,
then for some parameter choice $ \beta = \betast $
we have
%% with $\betast \to \delta $, we have
%
\begin{align}
& \normix{\ualpauxst - \ust} = \Landauno{\delta^{\frac{p}{p+\mya}}},
\quad
\norma{\ualpauxst - \ust}
= \Landauno{\delta}, \nonumber \\
& \normone{\ualpauxst - \ubar} = \Landauno{\delta^{-\frac{1-\pp}{\pp+\mya}}}
\as \delta \to 0.
\label{eq:auxel-2-n}
\end{align}

\item (Low order smoothness)
If $ \ust \in \domain(\log \G) $, then for some parameter choice $ \beta = \betast $
we have
\begin{align}
& \normix{\ualpauxst - \ust} = \Landauno{\loginv{\tfrac{1}{\delta}}},
\quad
\norma{\ualpauxst - \ust}
= \Landauno{\delta}, \nonumber \\
& \normone{\ualpauxst - \ubar} = \Landauno{\delta^{-\frac{1}{\mya}} \loginv[(1+\frac{1}{a})]{\tfrac{1}{\delta}}}
\as \delta \to 0.
\label{eq:auxel-2-c}
\end{align}
\end{mylist_indent}
For each of the three cases, the parameter choice $ \beta = \betast $ is specified in the proof.
%%respectively.
\end{lemma}
\proof
We consider the following choices of $ \betast $:
\begin{mylist_indent}
\item In case of no explicit smoothness, one may choose
$ \betast = c \delta^{\lfrac{1}{\mya}} $.
\item In case of H\"older smoothness, one can choose
$ \betast = c \delta^{\frac{1}{\pp+\mya}} $.
\item In case of low order smoothness, we consider
%%%$ \betast = c \mychi{-\frac{1}{a}}{\frac{1}{a}}(\delta) $
$ \betast = c \kla{\delta \myphibrez{\delta}}^{1/\mya} $
for $ 0 < \delta < \delta_0 $, with
$ \delta_0 $ sufficiently small.
\end{mylist_indent}
Here, $ c > 0 $ denotes an arbitrary constant factor.
The first two statements follow as an easy consequence of Lemma \ref{th:auxel_1}.
The statement on the low order case is also a consequence of Lemma \ref{th:auxel_1}. In this case, however, below we consider some details.
For this purpose, we will make use of the notation
$ \myphi(t) = \myphib{t} $
introduced in Section~\ref{low_order_preparations}.
%
%%\begin{myenumerate}
%%\item
We first note that for some constant $ c_1 > 0 $, we have
% yyyy
\begin{align}
\myphi(\betast) \le c_1 \myphi(\delta), \quad 0 < \delta \le \delta_0,
\label{eq:auxel-2-ca}
\end{align}
%zzzz
which in fact follows easily from the two estimates
given in items \ref{it:low_oder_prep_d} and \ref{it:low_oder_prep_e}
%%%\refeq{phi-prop1} and \refeq{phi-prop2}.
introduced in Section \ref{low_order_preparations}.
%
%%%\begin{align*}
%%\norma{\ualpauxst - \ust}
%%%\myphi(\betast) \le c_2 \myphi(\mychi{-\frac{1}{a}}{\frac{1}{a}}(\delta))
%%%\le c_1 c_2 \myphi(\delta).
%%\end{align*}
%
%%\item
The three given estimates for the low order case are now consequences
of Lemma \ref{th:auxel_1} and estimate \refeq{auxel-2-ca}.
For $ \normix{\ualpauxst - \ust} $ this is immediate, and in addition,
we also obtain the following:
\begin{align*}
\norma{\ualpauxst - \ust} & \le c_2 \myphi(\betast) \betast^a
\le c_3 \myphi(\delta) \kla{\myphi(\delta)^{-1} \delta}^{a/a}
= c_3 \delta, \\
\normone{\ualpauxst - \ubar} & \le c_4 \myphi(\betast) \betast^{-1}
\le c_5 \myphi(\delta) \kla{\myphi(\delta) \delta^{-1}}^{1/\mya}
= c_5 \myphi(\delta)^{1+\frac{1}{\mya}} \delta^{-1/\mya},
\end{align*}
where $ c_2, \ldots, c_5 $ denote appropriately chosen constants.
This completes the proof of the lemma.
%%\end{myenumerate}
\proofend
\subsection{Proof of Theorem \ref{th:discrepancy_main_result}}
\label{proof_main:theorem}
This section is devoted to the proof of our main result, Theorem \ref{th:discrepancy_main_result}.
As a basic ingredient, we need to provide
reasonable estimates of the two terms
$ \norma{\upardeldel - \ust } $ and
$ \normone{\upardeldel - \ubar } $.
We start with the estimation of the former one.
%%We first provide an elementary estimate of $ \norma{\upardeldel - \ust } $.
%
\begin{lemma}
\label{th:upardel_norma_lemma}
\mainassump
We then have
\begin{align*}
\norma{\upardeldel - \ust } = \Landauno{\delta} \as \delta \to 0.
\end{align*}
\end{lemma}
\proof  From the choice of $ \pardel$ and estimate
\refeq{normequiv_b}, it follows that
\begin{align}
\label{eq:upardel_norma_lemma}
& \ca \norma{\upardeldel -\ust} \le \normyps{\myF{\upardeldel} - \fst } \le
\normyps{\myF{\upardeldel} - \fdel } + \delta \le (b+1)\delta
\end{align}
for $ \delta > 0 $ small enough. Note that the upper bound presented at the end of
\refeq{upardel_norma_lemma} guarantees that
estimate \refeq{normequiv_b} is applicable with $ u = \upardeldel $
for $ \delta $ small enough. This concludes the proof.
\proofend
Our next goal is to provide appropriate estimates for
$ \normone{\upardeldel - \ubar } $. This requires some preparations.
For this purpose, we recall the definition from \refeq{mykapdef},
this is
%% $\er $ is defined as follows:
%%\begin{align*}
$ \mykap = \frac{1}{\fab} $.
%%\end{align*}
%
%%In addition, we use the notation
%
%%\begin{align*}
%%\mykap \defeq \frac{1}{\fab}.
%%\end{align*}
%
\begin{lemma}
\label{th:upardel_lemma}
\mainassump
There exists some $ \para_0 > 0 $ such that for $ 0 < \para \le \para_0 $ and each $\delta>0$, we have
\begin{align*}
\max\{\normyps{\myF{\upardel} - \fdel }, \, \para^{1/\myr} \normone{\upardel-\ubar}\}
\le \myfa(\para)\para^{\mykap a } + \er \delta,
\end{align*}
where the constant $ \er $ is given by \refeq{er}. In addition,
$\myfa(\para)$ is a bounded function which satisfies the following:
\begin{mylist_indent}
\item (No explicit smoothness)
If $ \ust \in \overline{\R(G)} $, then $ \myfa(\para) \to 0 $ as $ \para \to 0 $.
\item (H\"older smoothness)
If $ \ust \in \ix_\pp $ for some $ 0 < \pp \le 1 $, then
$ \myfa(\para) = \Landauno{\para^{\mykap p}} $ as $ \para \to 0 $.
\item (Low order smoothness)
If $ \ust \in \domain(\log \G) $, then
$ \myfa(\para) = \Landauno{\loginv{\frac{1}{\para}}} $ as $ \para \to 0 $.
\end{mylist_indent}
\end{lemma}
%
%%In addition, the constant $\er $ is defined as follows:
%%$$  \er = \left\{\begin{array}{ll} 1,
%%& \textup{if } r \ge 1, \\
%%2^{-1+1/r}& \textup{otherwise}.
%%\end{array}\right. $$
%
\proof
Let $ \ust \in \overline{\R(G)} $.
For auxiliary elements of the form \refeq{uaux-def},
with saturation $ \pp_0 \ge 1 + a $, we choose
\begin{align}
\parb=\parb(\para) = \para^{\mykap}.
\label{eq:parb-def}
\end{align}
For $ \para > 0 $ small enough, say $ 0 < \para \le \para_0 $, we have
$ \ualpaux \in \Dset $ because of Lemma~\ref{th:auxel_1} and since moreover $\ust$ is assumed to be an interior point of $\DF$. We thus have
\begin{align}
& \kla{\normypsqua{\myF{\upardel} - \fdel } + \para \normonequa{\upardel-\ubar}}^{1/\myr}
\le \kla{\normypsqua{\myF{\ualpaux} - \fdel } + \para \normonequa{\ualpaux-\ubar}}^{1/\myr}
\nonumber \\
& \quad \le e_r (\normyps{\myF{\ualpaux} - \fdel } + \para^{1/\myr} \normone{\ualpaux-\ubar}) \nonumber \\
& \quad
\le e_r(\normyps{\myF{\ualpaux} - \fst } + \para^{1/\myr} \normone{\ualpaux-\ubar} + \delta).
\label{eq:upardel_lemma_a}
\end{align}
%
%%%tttt
The first term on the \rhs of the latter estimate can be written as
\begin{align}
\normyps{\myF{\ualpaux} - \fst }
\le \cb \norma{\ualpaux - \ust}
= \cb g_2(\parb)\parb^{a}
= \cb g_2(\para^{\mykap})\para^{\mykap a}.
\label{eq:upardel_lemma_b}
\end{align}
%
%%%where $ g_2 $ corresponds to \refeq{auxel-b}.
The estimate in \refeq{upardel_lemma_b} is a consequence of estimate \refeq{normequiv_a}, which is applicable with $ u = \ualpaux $ for $ \para $ small enough, and without loss of generality we may assume that small enough means $ \para \le \para_0 $ by choosing $ \para_0 $ sufficiently small in the beginning. The first identity in \refeq{upardel_lemma_b} follows from representation \refeq{auxel-b} in Lemma \ref{th:auxel_1}.

The second term on the \rhs of the estimate \refeq{upardel_lemma_a}
can be represented as follows:
\begin{align*}
\para^{1/r} \normone{\ualpaux-\ubar}
=
\para^{1/r} g_3(\parb)\parb^{-1}
= g_3(\para^{\mykap})\para^{\mykap a},
\end{align*}
based on \refeq{auxel-c} of Lemma \ref{th:auxel_1}.
As a consequence, the estimate of Lemma \ref{th:upardel_lemma} holds, if the
function $ \myfa $ is chosen as
%%This yields the function
$$ \myfa(\para):=  \er(\cb g_2(\para^\mykap)+g_3(\para^\mykap)) \for \para \le \para_0 \,. $$
The asymptotic behavior of the function $ \myfa $ stated in the lemma is an immediate consequence of
Lemma~\ref{th:auxel_1}.
This completes the proof of the lemma.
\proofend
As a consequence of the preceding lemma, we can derive reasonable lower bounds for the regularizing parameter $ \parst $ obtained by the discrepancy principle, which actually affects the stability
of the method.
%xxxx
\begin{corollary}
\label{th:upardel_cor}
\mainassump Let the parameter $ \para = \pardel $ be chosen according to the discrepancy principle.
%%Then we have
%%%There exist finite positive constants $ \delta_0 $
%%%such that for $0 \le \delta \le \delta_0$ the following holds:
%
\begin{mylist_indent}
\item (No explicit smoothness)
If $ \ust \in \overline{\R(G)} $, then
%%%$ \pardel^{-\mykap} = \landauno{\delta^{-\lfrac{1}{\mya}}} \to 0 $ as $ \delta \to 0 $.
$ \pardel^{-\mykap \mya} = \landauno{\delta^{-1}} $ as $ \delta \to 0 $.
\item (H\"older smoothness)
If $ \ust \in \ix_\pp $ for some $ 0 < \pp \le 1 $, then
$ \pardel^{-\mykap(\pp + \mya)} = \Landauno{\delta^{-1}} $ as $ \delta \to 0 $.
\item (Low order smoothness)
If $ \ust \in \domain(\log \G) $, then
$ \pardel^{-\mykap\mya} = \Landauno{\delta^{-1}\loginv{\tfrac{1}{\delta}}} $ as $ \delta \to 0 $.
\end{mylist_indent}
\end{corollary}
\proof
%%%Without loss of generality, we may assume that $ \pardel \le \alpha_0 $ holds, since
We first note that parameters $ \pardel $ which stay away from zero
can easily be treated in each of the three cases.
%%that we do not cause any problems in the analysis.
Note also that this is a related to a degenerated case, and it includes the case $ \pardel = \infty $.

In the following, we thus may assume that $ \pardel \le \alpha_0/c $
and thus $ \pardelb \le \alpha_0 $ hold, where
$ \alpha_0 $ is given by Lemma \ref{th:upardel_lemma}, and
the constant $ c $ and the parameter $ \pardelb $ are introduced by the discrepancy principle \refeq{parameter-choice}.
Lemma \ref{th:upardel_lemma} then implies
%%%\begin{align*}
$ b \delta  \le \normyps{\myF{\upardeldelb} - \fdel }
\le \myfa(\pardelb)\pardelb^{\mykap a } + \er \delta $
and thus
\begin{align}
(b-\er) \delta  \le \myfa(\pardelb)\pardelb^{\mykap \mya}.
\label{eq:upardel_cor_a}
\end{align}
The statements of the corollary for the two cases ``no explicit smoothness''
 and ``H\"older smoothness'' now easily follow from the properties on
the function $\myfa $ presented in Lemma \ref{th:upardel_lemma},
respectively.
%%%
Low order smoothness is considered next.
In this case, without loss of generality we may assume that $ \para_0 < 1 $.
%%% is chosen so small that
%%$ \alpha_0 < 1 $.
%%%$ \para_0 < 1/c $, where $ c $ denotes the constant from
%%%\refeq{parameter-choice}.
%%
Estimate \refeq{upardel_cor_a} then means
\begin{align*}
c_1 \delta  \le \myphib{\pardelb} \pardelb^{\mykap \mya} = \mychi{1}{\mykap \mya}(\pardelb),
\end{align*}
where $ c_1 >0 $ denotes a constant,
and the notation from Section \ref{low_order_preparations} is used again.
From item \ref{it:low_oder_prep_a}
%%%and \ref{it:low_oder_prep_e}
of that section, we now easily obtain
$ \mychiinv{1}{\mykap \mya}(c_1 \delta) \le \pardelb $,
with $ \delta > 0 $ small enough.
%
%%%qqqqq
This provides the basis for the following estimates, which also utilize items~\ref{it:low_oder_prep_c} and \ref{it:low_oder_prep_e} from Section \ref{low_order_preparations}:
\begin{align*}
c \pardel & \ge \pardelb \ge \mychiinv{1}{\mykap \mya}(c_1 \delta)
\ge c_2 \mychi{-1}{1}(c_1 \delta)^{\lfrac{1}{(\mykap \mya)}}
\\
& = c_3 (\delta \myphibrez{c_1 \delta})^{\lfrac{1}{(\mykap \mya)}}
\ge c_4 (\delta \myphibrez{\delta})^{\lfrac{1}{(\mykap \mya)}},
\end{align*}
where $ c_2, c_3 $ and $ c_4 $ denote appropriately chosen finite constants, and $ \delta $ is again sufficiently small.
A simple rearrangement yields the statement on low order smoothness.
\proofend
Below, we present suitable estimates for
$ \normone{\upardeldel - \ubar } $.
\begin{corollary}
\label{th:upardel_cor_b}
\mainassump Let the parameter $ \para = \pardel $ be chosen according to the discrepancy principle. Then
%%%there exist a finite positive constant $ \delta_0 $
%%%such that for $0 \le \delta \le \delta_0$
the following holds:
\begin{mylist_indent}
\item (No explicit smoothness)
If $ \ust \in \overline{\R(G)} $, then
$ \normone{\upardelst - \ubar} = \landauno{\delta^{-\lfrac{1}{a}}} $
as $ \delta \to 0 $.
\item (H\"older smoothness)
If $ \ust \in \ix_\pp $ for some $ 0 < \pp \le 1 $, then
$ \normone{\upardelst - \ubar} = \Landauno{\delta^{-\frac{1-\pp}{\pp+\mya}}} $
as $ \delta \to 0 $.
%xxx
\item (Low order smoothness)
If $ \ust \in \domain(\log \G) $, then
$ \normone{\upardelst - \ubar} = \Landauno{\delta^{-\frac{1}{\mya}} \loginv[(1+\frac{1}{a})]{\tfrac{1}{\delta}}} $
as $ \delta \to 0 $.
\end{mylist_indent}
\end{corollary}
%xxxx
\proof
For parameters $ \pardel $ staying away from the origin, say $ \pardel \ge \alpha_1 > 0 $, the statements of the corollary follow immediately, since $ \normone{\upardelst - \ubar} $ stays bounded then,
%% in each of the three case
as can be seen from the following computations:
\begin{align*}
\alpha_1^{1/\myr} \normone{\upardelst - \ubar}
& \le \pardel^{1/\myr} \normone{\upardelst - \ubar}
\le \pad{\upardelst}^{1/\myr} \le
\pad{\ubar}^{1/\myr} \\
& = \normyps{\myF{\ubar} - \fdel }
\le \normyps{\myF{\ubar} - \fst } + \delta.
\end{align*}
%
%%%Without loss of generality, we may assume that $ \pardel \le \alpha_0 $ holds, since
%%%For parameters $ \pardel > \alpha_0 $ -- a degenerate case which includes $ \pardel = \infty $ -- the statements of the corollary follow  xxxxx
%
Therefore, in the following we may assume
$ \pardel \le \alpha_0 $, where $ \alpha_0 $ is given by Lemma \ref{th:upardel_lemma}.
The same lemma then implies
$ \parst^{1/\myr} \normone{\upardeldel-\ubar}
\le \myfa(\parst)\parst^{\mykap a } + \er \delta $ and thus
\begin{align}
\normone{\upardeldel-\ubar}
\le \myfa(\parst)\,\parst^{-\mykap } + \er \frac{\delta}{\parst^{1/\myr}},
\label{eq:upardel_cor_b1}
\end{align}
where we make use of the identity $ \mykap \mya - \frac{1}{r} = -\mykap $.
The statements of the corollary now follow by considering the two terms on the \rhs of
\refeq{upardel_cor_b1} separately, respectively.
For the two cases ``no explicit smoothness'' and ``H\"older smoothness'', this follows from the corresponding estimates from Lemma \ref{th:upardel_lemma} and Corollary \ref{th:upardel_cor}. The proof is straightforward, and details thus are omitted here.

Below we present some details for the low order smoothness case.
In this case, the estimate of
the function $ \myfa $ given by Lemma \ref{th:upardel_lemma}
yields
%%% \refeq{upardel_cor_b1}
%
\begin{align}
\normone{\upardeldel-\ubar}
\le c_1 \myphib{\parst}\parst^{-\mykap } + \er \frac{\delta}{\parst^{1/\myr}},
\label{eq:upardel_cor_b2}
\end{align}
%
%uuuu
for some suitable finite constant $ c_1 $.
%%%Both terms on the \rhs of \refeq{upardel_cor_b2} can be easily suitable estimated by using the lower estimate of $ \parst $
We now  can proceed
%%%with the \rhs of \refeq{upardel_cor_b2} can be easily suitable estimated
by utilizing the lower estimate of $ \parst $
given by Corollary \ref{th:upardel_cor}, i.\,e.,
\begin{align}
c_2 ( \delta \mylog{\tfrac{1}{\delta}})^{1/(\mykap\mya)}
\le \pardel
\for 0 < \delta \le \delta_1,
\label{eq:upardel_cor_b3}
\end{align}
with some constant $ c_2 $, and $ \delta_1 $ is chosen small enough.
From \refeq{upardel_cor_b3}, the second term on the \rhs of \refeq{upardel_cor_b2} can be suitable estimated in a straightforward manner, and we omit the details.
%
%%the lower estimate of $ \parst $
%
%%%$ \pardel^{-\mykap\mya} = \Landauno{\delta^{-1}\loginv{\tfrac{1}{\delta}}} \to 0 $ as $ \delta \to 0 $.
We next consider the first term on the \rhs of \refeq{upardel_cor_b2}.
Without loss of generality, in the following we may assume that $ \alpha_0 < 1 $ considered in the beginning of the proof is chosen so small such that
%$ \mychi{1}{-\mykap}(\para) = \mychi{-1}{\mykap}^{-1}(\para) $
the function
$ \kla{\mychi{-1}{\mykap}(\para)}^{-1} =
\myphib{\para}\para^{-\mykap } $
%%% on the \rhs of \refeq{upardel_cor_b3}
%% $ \mychi{-1}{\mykap} $
is monotonically decreasing for $ 0 < \para \le \para_0 $, cf.~item~\ref{it:low_oder_prep_b} in Section~\ref{low_order_preparations}.
From \refeq{upardel_cor_b2}, \refeq{upardel_cor_b3} and items \ref{it:low_oder_prep_b} in Section \ref{low_order_preparations}, we then obtain
\begin{align*}
%%& \myphib{\parst}\parst^{-\mykap } \le c_1 s_1 s_2, \\
& \myphib{\parst}\parst^{-\mykap } \le
c_3 \kla{\delta \mylog{\tfrac{1}{\delta}}}^{-1/\mya} \sigma,
%%%\quad \textup{ with }
\quad
\sigma \defeq
(-\mylog\kla{{c_2 ( \delta \mylog{\delta})^{1/(\mykap\mya)}}})^{-1},
%%= \myphi{\kla{c ( \delta \mylog{\tfrac{1}{\delta}})^{1/(\mykap\mya)}}},
%%%= \myphi\kla{{c \mychi{-\frac{1}{\mykap\mya}}{\frac{1}{\mykap\mya}}(\delta)}},
%%\quad s_2 \defeq \kla{\delta \mylog{\tfrac{1}{\delta}}}^{-1/\mya}.
%
%%\kla{c ( \delta \mylog{\tfrac{1}{\delta}})^{1/(\mykap\mya)}}^{-\mykap }
%%c_1 \myphi(\parst)\parst^{-\mykap }
%%= c_2 \mychirez{-1}{\mykap}{\parst}.
%%%\label{eq:upardel_cor_b4}
\end{align*}
for some constant $ c_3 $.
From items \ref{it:low_oder_prep_d} and \ref{it:low_oder_prep_e} in Section
\ref{low_order_preparations}, it follows that
\begin{align*}
%%\sigma = \myphi\kla{{c \mychi{-\frac{1}{\mykap\mya}}{\frac{1}{\mykap\mya}}(\delta)}}
\sigma = \myphi\kla{c_2 \mychi{-1}{1}(\delta)^{1/(\mykap\mya)}}
\le c_4 \myphi(\delta) = c_4 \myphib{\delta},
\end{align*}
for some constant $ c_4 $.
The statement in the third item of the corollary now follows.
%%%\begin{align}
%%%\mychirez{-1}{\mykap}{\parst}
%%& \le \mychirez{-1}{\mykap}{c \mychi{-1}{1}(\delta)^{1/\mykap \mya}}
%%& \le \mychirez{-1}{\mykap}{c ( \delta \mylog{\tfrac{1}{\delta}})^{1/(\mykap\mya)})}
%%%\end{align}
%
%%% zzz
\proofend
%
%
\begin{comment}
\begin{corollary}
\label{th:upardel_cor}
\mainassump
There exist finite positive constants $ \para_0, \delta_0 $
and $K_2$
such that for $ 0 < \para \le \para_0 $ and each $0 \le \delta \le \delta_0$, we have
%
$$ \norma{\upardel -\ust} \le f_3(\para)\para^{\mykap} + K_2 \delta.$$
Here $f_3(\para), 0 < \para \le \para_0 $, is a bounded function which satisfies the following:
%
\begin{mylist_indent}
\item (No explicit smoothness)
If $ \ust \in \overline{\R(G)} $, then $ f_3(\para) \to 0 $ as $ \para \to 0 $.
\item (H\"older smoothness)
If $ \ust \in \ix_\pp $ for some $ 0 < \pp \le 1 $, then
$ f_3(\para) = \Landauno{\para^{\mykap p}} $ as $ \para \to 0 $.
\item (Low order smoothness)
If $ \ust \in \domain(\log \G) $, then
$ f_3(\para) = \Landauno{\loginv{\frac{1}{\para}}} $ as $ \para \to 0 $.
\end{mylist_indent}
\end{corollary}
%
\proof
Let $ \para $ and $ \delta $ be small enough, say
$ 0 < \para \le \para_0 $ and $ 0 < \delta \le \delta_0 $.
From estimate \refeq{normequiv_b} and Lemma \ref{th:upardel_lemma},
it then follows
%
\begin{align*}
& \ca \norma{\upardel -\ust} \le \normyps{\myF{\upardel} - \fst } \le
\normyps{\myF{\upardel} - \fdel } + \delta
\le f_2(\para)\,\para^{\mykap a} + (1+\er)\delta.
\end{align*}
%
The assertion of the corollary now follows by setting $f_3(\para):=\tfrac{f_2(\para)}{\ca}$ and $K_2:=\tfrac{1+\er}{\ca}.$
\proofend
\end{comment}
%
We are now in a position to present a proof of the main result of this paper.
\begin{proof}[Proof of Theorem \ref{th:discrepancy_main_result}]
We start with an elementary error estimate $ \norm{\upardeldel -\ust} $
%%%can be estimated as follows
utilizing the auxiliaries,
%% the following series of error estimates.
%%% in combination with
%%\refeq{auxel-a} from Lemma~\ref{th:auxel_1}, we obtain
%
\begin{align}
\normix{\upardeldel -\ust}\le
\normix{\upardeldel -\ualpauxst}+\normix{\ualpauxst-\ust},
%%%=\normix{\upardel -\ualpaux}+g_1(\para^\mykap),
\label{eq:upardel_prop-a}
\end{align}
where $ \betast $ is given by
Lemma~\ref{th:auxel_2}.
%%The second term
The error of the auxiliaries
on the \rhs of \refeq{upardel_prop-a}
can be properly estimated using Lemma \ref{th:auxel_2}.
Below we consider the term $ \normix{\upardel -\ualpaux} $ in more detail.
From the interpolation inequality \refeq{interpol2},
it follows
\begin{align}
\label{eq:upardel_prop-b}
\normix{\upardeldel -\ualpauxst} & \le
c_1 \norma{\upardeldel -\ualpauxst}^{\lfrac{1}{(a+1)}}
\normone{\upardeldel -\ualpauxst}^{\lfrac{a}{(a+1)}},
\end{align}
where $ c_1 $ denotes some finite constant not depending on $ \delta $.
The first term on the \rhs of estimate \eqref{eq:upardel_prop-b}
can be estimated by using
Lemmas~\ref{th:auxel_2} and
\ref{th:upardel_norma_lemma}.
%% Corollary~\ref{th:upardel_cor} and Lemma~\ref{th:auxel_1} in the following manner.
Precisely, we find
%% with
%%$$f_4(\para):=f_3(\para)+g_2(\para^\mykap), \qquad
%%f_5(\para):=f_2(\para)+g_1(\para^\mykap)$$
%
the estimates
\begin{align*}
\norma{\upardelst -\ualpauxst}
& \le
\norma{\upardelst -\ust} + \norma{\ualpauxst - \ust}
= \Landauno{\delta} \as \delta \to 0,
%%% \le f_4(\alpha)\para^{\mykap a}+K_2\delta,\\
\end{align*}
so that estimate \refeq{upardel_prop-b} simplifies to
\begin{align}
\label{eq:upardel_prop-c}
\normix{\upardeldel -\ualpauxst} & \le
c_2 \delta^{\lfrac{1}{(a+1)}}
\normone{\upardeldel -\ualpauxst}^{\lfrac{a}{(a+1)}},
\end{align}
where $ c_2 $ denotes some finite constant independent of $ \delta $.
The last factor on the \rhs of \refeq{upardel_prop-c} is estimated next, and for this purpose, we make use of the following elementary estimate,
%%The second term on the \rhs of \refeq{upardel_prop-c} is estimated as follows,
%
\begin{align}
\label{eq:upardel_prop-d}
\normone{\upardeldel - \ualpauxst}
\le
\normone{\upardeldel - \ubar} + \normone{\ualpauxst - \ubar}.
\end{align}
We now proceed with the estimation of the \rhs of \refeq{upardel_prop-d}
by distinguishing our different smoothness assumptions.
%%% and we make constant use of
%%Lemmas~\ref{th:auxel_2} for this.
%%%
%
\begin{myenumerate}
\item
For $ \ust \in \overline{\R(G)} $ (no explicit smoothness),
from
estimate \refeq{upardel_prop-d}, Lemma \ref{th:auxel_2} and Corollary \ref{th:upardel_cor_b}
we obtain
\begin{align*}
\normone{\upardeldel - \ualpauxst}
\le
%%%\normone{\upardeldel - \ubar} + \normone{\ualpauxst - \ubar}
%%%=
\landauno{\delta^{-\lfrac{1}{a}}}
+ \landauno{\delta^{-\lfrac{1}{a}}}
=
\landauno{\delta^{-\lfrac{1}{a}}},
\end{align*}
%%%Introducing $ f_6(\para):= \max\{ f_4(\para), f_5(\para) \} $
%%and $ K_3 := \max\{K_2, 1\}, \ K_1 = c_3 K_3 $, we obtain
%
and estimate \refeq{upardel_prop-c}
%%% and \refeq{upardel_prop-d}
%%%as well as Lemma~\ref{th:auxel_2}
then gives
\begin{align*}
\normix{\upardelst -\ualpauxst}
 \le c_2 \delta^{\tfrac{1}{a+1}}
\landauno{\delta^{-\tfrac{1}{a+1}}}
%%= \landauno{1}
\to 0 \as \delta \to 0.
%%%& \le c_3 \left(f_6(\alpha)\,\para^{\mykap a}+K_3\delta \right)^{1/(a+1)}\,\left(\para^{-\lfrac{1}{\myr}}\left(f_6(\para)\, \para^{\mykap a}+ K_3\delta\right) \right)^{a/(a+1)}.
\end{align*}
This result in combination with estimate \refeq{upardel_prop-a}
and Lemma~\ref{th:auxel_2}
yields
$
\normix{\upardeldel -\ust} \to 0 $ as $ \delta \to 0 $. This is the first statement of
Theorem \ref{th:discrepancy_main_result}.
%
%%\begin{align}
%%\normix{\upardeldel -\ust}\le
%%\normix{\upardeldel -\ualpauxst}+\normix{\ualpauxst-\ust},
%%%=\normix{\upardel -\ualpaux}+g_1(\para^\mykap),
%%\label{eq:upardel_prop-a}
%%%\end{align}
%
\item (H\"older smoothness)
If $ \ust \in \ix_\pp $ for some $ 0 < \pp \le 1 $, then
from
estimate \refeq{upardel_prop-d}, Lemma \ref{th:auxel_2} and Corollary \ref{th:upardel_cor_b}
we obtain
\begin{align*}
\normone{\upardeldel - \ualpauxst}
\le  \Landauno{\delta^{-\frac{1-\pp}{\pp+\mya}}} +
\Landauno{\delta^{-\frac{1-\pp}{\pp+\mya}}}
=
\Landauno{\delta^{-\frac{1-\pp}{\pp+\mya}}},
\end{align*}
%%%as $ \delta \to 0 $.
and estimate \refeq{upardel_prop-c}
%%% and \refeq{upardel_prop-d}
%%%as well as Lemma~\ref{th:auxel_2}
then gives
\begin{align*}
\normix{\upardelst -\ualpauxst}
 \le
c_2 \delta^{\frac{1}{a+1}} \Landauno{\delta^{-\frac{1-\pp}{\pp+\mya}\frac{\mya}{\mya+1}}}
= \Landauno{\delta^{\frac{\pp}{\pp+\mya}}}
%%= \landauno{1}
\as \delta \to 0.
%%%& \le c_3 \left(f_6(\alpha)\,\para^{\mykap a}+K_3\delta \right)^{1/(a+1)}\,\left(\para^{-\lfrac{1}{\myr}}\left(f_6(\para)\, \para^{\mykap a}+ K_3\delta\right) \right)^{a/(a+1)}.
\end{align*}
This estimate combined with estimate \refeq{upardel_prop-a}
and Lemma~\ref{th:auxel_2}
yields
$
\normix{\upardeldel -\ust}
= \Landauno{\delta^{\frac{\pp}{\pp+\mya}}}
$ as $ \delta \to 0 $. This is the second statement of
Theorem \ref{th:discrepancy_main_result}.

\item (Low order smoothness)
If $ \ust \in \domain(\log \G) $, then
from
estimate \refeq{upardel_prop-d}, Lemma \ref{th:auxel_2} and Corollary \ref{th:upardel_cor_b}
we obtain
\begin{align*}
\normone{\upardeldel - \ualpauxst}
= \Landauno{\delta^{-\frac{1}{\mya}} \loginv[(1+\frac{1}{a})]{\tfrac{1}{\delta}}},
\end{align*}
%%%as $ \delta \to 0 $.
and estimate \refeq{upardel_prop-c}
%%% and \refeq{upardel_prop-d}
%%%as well as Lemma~\ref{th:auxel_2}
then gives
\begin{align*}
\normix{\upardelst -\ualpauxst}
 \le
c_2 \delta^{\frac{1}{\mya+1}}
\Landauno{\delta^{-\frac{1}{\mya+1}} \loginv{\tfrac{1}{\delta}}}
%%\loginv[(1+\frac{1}{a})]{\tfrac{1}{\delta}}}
=
\Landauno{\loginv{\tfrac{1}{\delta}}} \as \delta \to 0.
\end{align*}
This estimate in combination with  \refeq{upardel_prop-a}
and Lemma~\ref{th:auxel_2}
yields
$
\normix{\upardeldel -\ust}
= \Landauno{\loginv{\tfrac{1}{\delta}}} $ as $ \delta \to 0 $. This is the third and final statement of
Theorem \ref{th:discrepancy_main_result}.
\end{myenumerate}
%%%From the latter estimate and \refeq{upardel_prop-a}, the theorem now immediately follows.
%
\end{proof}

\section*{Acknowledgment}
This paper was created as part of the authors' joint DFG-Project No.~453804957 supported by the German Research Foundation under grants PL 182/8-1 (Chantal Klinkhammer, Robert Plato) and  HO 1454/13-1 (Bernd Hofmann).
\bibliography{HKP}

\end{document}